\documentclass[12pt,a4paper]{amsart}
\usepackage{geometry}
\usepackage{amsmath,amssymb}
\usepackage{amsfonts}
\usepackage{eucal}
\usepackage{amsthm}
\usepackage{mathrsfs}
\usepackage{ascmac}
\usepackage{mathtools}
\numberwithin{equation}{section}
\usepackage{comment}
\usepackage{bm}
\usepackage{fancybox,varwidth,color}
\usepackage{comment}

\newcommand{\jap}[1]{\langle #1 \rangle}

\def\a{\alpha}
\def\b{\beta}
\def\c{\gamma}
\def\d{\delta}
\def\e{\varepsilon}
\def\f{\varphi}
\def\g{\psi}

\def\k{\kappa}
\def\l{\lambda}
\def\m{\mu}

\def\s{\sigma}

\def\x{\xi}
\def\y{\eta}

\renewcommand{\L}{\Lambda}

\def\tor{\mathbb{T}}

\newcommand{\Op}{\mathrm{Op}}

\def\re{\mathbb{R}}

\def\ze{\mathbb{Z}}

\def\pa{\partial}

\renewcommand{\Re}{\text{{\rm Re}\;}}
\renewcommand{\Im}{\text{{\rm Im}\;}}

\newcommand{\supp}{\text{{\rm supp}\;}}

\newcommand{\Ker}{\text{{\rm Ker}\;}}
\newcommand{\Ran}{\text{{\rm Ran}\;}}

\newcommand{\sgn}{\text{{\rm sgn}\;}}

\newtheorem{thm}{Theorem}[section]
\newtheorem{lem}[thm]{Lemma}
\newtheorem{prop}[thm]{Proposition}
\newtheorem{cor}[thm]{Corollary}
\newtheorem{defn}[thm]{Definition}

\theoremstyle{definition}

\newtheorem{example}{Example}

\theoremstyle{remark}
\newtheorem{rem}[thm]{Remark}

\title[]%
{Equivalence of classical and quantum completeness for real principal type operators on the circle}
\author[]%
{Kouichi Taira}
\address{Faculty of Mathematics, Kyushu University, 744, Motooka, Nishi-ku, Fukuoka, Japan}
\email{taira.kouichi.800@m.kyushu-u.ac.jp}

\begin{document}
\maketitle

\begin{abstract}
In this article, we prove that the completeness of the Hamilton flow and essential self-adjointness are equivalent for real principal type operators on the circle. Moreover, we study spectral properties of these operators. The proof is based on the construction of eigenfunctions with non-real eigenvalues which is well-known in scattering theory. Moreover, the relationship between scattering theory and the essential self-adjointness is explained.
\end{abstract}

\section{Introduction}

\subsection{Motivation and history}

In this note, we study the relationship between the classical completeness and the quantum completeness. For a  differential operator $P=\sum_{|\a|\leq m}a_{\a}(x)D_x^{\a}$ (or more generally a pseudodifferential operator) on a manifold $M$, we say that:
\begin{itemize}
\item The system is classically complete if the Hamilton vector field $H_p=\pa_{\x}p\cdot \pa_x-\pa_xp\cdot \pa_{\x}$ on $T^*M$ generated by its principal symbol $p(x,\x)=\sum_{|\a|= m}a_{\a}(x)\x^{\a}$ is complete;

\item The system is classically complete if $P$ with domain $C_c^{\infty}(M)$ is essentially self-adjoint, that is $P|_{C_c^{\infty}}(M)$ has a unique self-adjoint extension.
\end{itemize}
Sometimes the full symbol $\sum_{|\a|\leq  m}a_{\a}(x)\x^{\a}$ should be considered in the former case.
The purpose of this paper is to construct a class of pseudodifferential operators on the one-dimensional torus such that the classical completeness is equivalent to the quantum completeness.
 As a corollary, some wave operators that were known to be classically incomplete are proved to be quantum incomplete.

It is believed that the classical completeness and the quantum completeness are closely related. This is because the essential self-adjointness of a differential operator $P$ is equivalent to the existence and uniqueness of solutions to a time-dependent Schr\"odinger equation
\begin{align*}
i\pa_tu+Pu=0,\quad u|_{t=0}\in L^2
\end{align*}
essentially due to Stone's theorem. On the other hand, the classical completeness means that the Hamiltonian equation of the symbol for $P$ with arbitrary initial data has unique time-global solutions. In other words, these completeness corresponds to the well-posedness of the master equations in classical mechanics and quantum mechanics respectively. See \cite[Volume II, Appendix X.I]{RS}, \cite{BMS}, and the references therein.

A natural problem is to find a class of operators whose classical and quantum completeness are equivalent to each other. In the recent paper \cite{CB}, it is conjectured that this is true for real-principal type operators on closed manifolds (see \cite[\S2, \S 14]{CB}) and gives an affirmative answer for wave operators associated with generic Lorentzian metrics on the two-dimensional torus. In this paper, we give an affirmative answer for one-dimensional real principal type operators.

There are lots of literatures about these equivalence for elliptic operators:
\begin{itemize}

\item The classical completeness is equivalent to the quantum completeness for a large class of one-dimensional Schr\"odinger operators $-\pa_x^2+V(x)$ on $(0,\infty)$ (\cite[Theorem X.9]{RS}). On the other hand, there are pathological examples of $V$, where these equivalence does not hold (\cite[Volume II, Examples 1,2]{RS}).

\item For a complete Riemannian manifold $(M,g)$, the Laplace-Beltrami operator $-\Delta_g$ is essentially self-adjoint on $C_c^{\infty}(M)$ (\cite{Ch,Ga,S}). In particular, the classical completeness implies the quantum completeness in this situation since the classical completeness is equivalent to the geodesic completeness;

\item The converse is not true. In fact, we consider the Riemmanian manifold $(\re^n\setminus \{0\},g_{\mathrm{std}})$, where $g_{\mathrm{std}}$ is the standard Riemmanian metric on the Euclidean space. Then $-\Delta$ is essentially self-adjoint on $C_c^{\infty}(\re^n\setminus\{0\})$ if and only if $n\geq 4$  (see \cite[after Theorem X.11]{RS}) although $\re^n\setminus \{0\}$ is not geodesically complete for all dimensions $n$.

\item  In \cite{BMS}, the essential self-adjointness of Schr\"odinger operators $-\Delta_g+V$ on non-compact Riemannian manifolds and other motivated examples is studied. 

\end{itemize}

On the other hand, the essential self-adjointness of non-elliptic operators, such as wave operators on Lorentzian manifolds, has been studied over the last decade:

\begin{itemize}
\item The essential self-adjointness of the wave operator on static spacetimes is verified by using separation of variables in \cite{DS1}. In \cite{DS2}, it was conjectured that this is also true on asymptotically static spacetime (a perturbation of a static spacetime) and this problem was solved in \cite{NT2} later. On spacetimes with compactly supported perturbation, \cite{WZ} gives its much simpler proof. 

\item In \cite{V2,NT,NT3}, it is shown that the wave operator is essentially self-adjoint on $C_c^{\infty}$ if the Lorentzian manifold is close to Minkowski near infinity under the null non-trapping condition.

\item 
On the contrary, the paper \cite{K} gives an example on a globally hyperbolic and geodesically complete Lorentzian metric on $\re^4$ whereas the corresponding wave operator is not essentially self-adjoint on $C_c^{\infty}(\re^4)$, which is motivated by the classical example of one-dimensional Schr\"odinger operator (see \cite[Example 2, Appendix X.I]{RS}).

\item In \cite{KK2}, Kassell-Kobayashi showed  the essential self-adjointness on the standard Lorentzian locally symmetric spaces (for the definition of the standard spaces, see \cite[\S1.1]{KK2}). Before that, they started to study spectral theory on Lorentzian locally symmetric spaces in \cite{KK}. They discovered examples of Lorentzian locally symmetric spaces where eigenvalues of the wave operators are invariant under the deformation of the lattices (the discrete subgroups).  For other motivated examples, see the survey article \cite{Ko}.

\end{itemize}
The essential self-adjointness of wave operators on such spaces are used to deduce the spectral action principle (\cite{DW,WZ}), which reveals the relationship between (a power of) the Feynman propagator and the scalar curvature. Here the former plays a significant role in Quantum Field Theory and the latter is important in the context of the Einstein-Hilbert action that represents the action for which the Euler-Lagrange equation gives rise to the Einstein equation.

We remark that the essential self-adjointness of non-elliptic operators was less studied before the above works. In fact, the problem is more difficult than the elliptic or subelliptic cases since most of previous works on essential self-adjointness crucially depend on the (sub)ellipticity, the positivity or the symmetry of the operators. Here we give a new insight from the perspective of scattering theory to tackle the essential self-adjointness problem on closed manifolds following author's previous work \cite{T} for the repulsive Schr\"odinger operators.
 The author hopes that the method and the idea given in the paper are one of the key ingredients to study the essential self-adjointness of non-elliptic operators.

In summary, the purpose of this paper is the following:
\begin{itemize}
\item To prove equivalence for classical and quantum completeness for real principal type operators on the circle $\tor=\re/2\pi \ze$. 
\item To give its proof in view of microlocal analysis and to apply the technique developed in \cite{T} to operators on closed manifolds.

\end{itemize}
Moreover, we study spectral properties of real principal type operators on the circle. Here we use the $L^2$-space as $L^2=L^2(\tor, dx)$ with the flat density $dx$. In addition, our method admits long-range perturbations $V$.

\subsection{Results on classical and quantum completeness}

Our main result of this paper is the following theorem. For the definition of the symbol class $\Psi^m_{\mathrm{phg}}$ and the principal symbol $\s$, see Section \ref{sectionmic}.

\begin{thm}\label{mainess}
Let $m\geq 0$ and $P\in \Psi^m_{\mathrm{phg}}$, which is symmetric on $C^{\infty}(\tor)$ with respect to the inner product of $L^2=L^2(\tor,dx)$.
Suppose that the principal symbol $p:=\s(P)$ is homogeneous of order $m$ and is real principal type in the sense of Definition \ref{pridef}.
Then the following conditions are equivalent:

\noindent$(i)$ The Hamilton vector field $H_p$ is complete as a vector field on $T^*\tor$.

\noindent$(ii)$ The operator $P+V$ is essentially self-adjoint on $C^{\infty}(\tor)$ for each symmetric operator $V\in \Psi^{m-\k}_{\mathrm{phg}}$ with $\k>\frac{1}{2}$.

\noindent$(iii)$ $m\leq 1$ or $p$ is elliptic in the sense of Definition \ref{pridef}.

\end{thm}

\begin{rem}
In $(ii)$, we only consider long-range perturbations of order $m-\k$ with $\k>\frac{1}{2}$. The only place where the condition $\k>\frac{1}{2}$ (instead of $\k>0$) is used is in the construction of an incoming WKB ansatz (Proposition \ref{WKBthm}). A more precise analysis like \cite[\S 4]{T} may allow us to deal with perturbations of order $m-\k$ with $\k>0$.
\end{rem}

The proof of $(i)\Leftrightarrow (iii)$, $(iii)\Rightarrow (ii)$ and $(ii)\Rightarrow (iii)$ are given in Subsections \ref{subseciandiii}, \ref{subseciiitoii} and \ref{subseciitoiii} respectively.

Now we give an example of $P$.

\begin{example}
Let $p(x,\x)=(\sin x)\x^2$ and $P=-\pa_x(\sin x\pa_x)$. Then $P$ satisfies the assumptions of Theorem \ref{mainess}. Since $m=2$ in this case and since $p$ is not elliptic, Theorem \ref{mainess} implies that $P$ is not essential self-adjoint on $C^{\infty}(\tor)$. More generally, $p(x,\x)=a(x)\x^2$ and $P=-\pa_x(a(x)\pa_x)$ satisfies Assumption of Theorem \ref{mainess}, where a real-valued smooth function $a(x)$ only has simple zeros.
\end{example}

\begin{rem}

$(1)$ In the expository note \cite{T2}, we give four different ways of the proof for non-essential self-adjointness of a simple operator $P=-\pa_x(\sin x \pa_x)$. The first one is the same method as here and the second one depends on another Fredholm problem inspired by the paper \cite{V1}.
The third one is to use the Fourier series and construct $L^2$-eigenfunctions associated with non-real parameters directly. The fourth proof is based on the integration by parts, which can be regarded as a boundary pairing formula like \eqref{boundrypair} below.

$(2)$
Let $p(x,\x)=(\sin^k x)\x^2$ and $P=-\pa_x(\sin^k x\pa_x)$ with $k\geq 2$. Then it follows that $p$ is not real principal type. Actually, as is proved in \cite{CB}, $H_p$ is complete and $P$ is essentially self-adjoint on $C^{\infty}(\tor)$.
\end{rem}

As a corollary, we show that the wave operators on basic examples (\cite{O}) of null incomplete closed Lorentzian manifolds are not essentially self-adjoint although the former example is already dealt with in \cite{CB}.

\begin{cor}\label{examplecor}
\noindent$(i)$
Let $\a\in \re\setminus \{0\}$ and $g_{\a}=-\a(\sin x) dy^2+2dxdy$ be a Lorentzian metric on $\tor^2_{(x,y)}$. Then the wave operator $\Box_{g_{a}}=2\pa_{x}\pa_y+\a\pa_{x}(\sin x\pa_x)$ with the metric $g_{\a}$ is not essential self-adjoint on $C^{\infty}(\tor^2)$.

\noindent$(ii)$ Let $(M,g)$ be the Clifton-Pohl torus (for its definition, see Section $\ref{Sectionproofcor}$ or \cite[\S7, Example 16]{O}). Then $\Box_g$ is not essentially self-adjoint on $C^{\infty}(M)$.
\end{cor}

The proof of this corollary is given in Section \ref{Sectionproofcor}.

\begin{rem}
$(1)$ The metric $g_{\a}=-\a(\sin  x) dy^2+2dxdy$ is very similar to the Schwarzschild metric with the Eddington-Finkelstein coordinate
\begin{align*}
-\left(1-\frac{r_g}{r} \right)dv^2+2dvdr+r^2(d\theta^2+\sin^2\theta d\f^2)
\end{align*}
near the even horizon $r=r_g$. In fact, $\sin  x$ has a simple zero at $x=0$ on one hand and $1-\frac{r_g}{r}$ also has a simple zero at $r=2r_g$ on the other hand.
In our proof, the lack of the essential self-adjointness of the operator $\Box_{g_{\a}}$ is due to the existence of zeros of $\sin  x$. Thus, it might suggest that the event horizon in the Schwarzschild spacetime breaks the essential self-adjointness of the wave operator there.
 
$(2)$
From the first part of the corollary, it turns out that
\begin{itemize}
\item A wave operator associated with the Lorentzian metric is not always essentially self-adjoint even on the closed manifold.
\item Essential self-adjointness of wave operators is not stable under a $C^0$-perturbation of the metric since $\Box_{g_0}$ (for $\a=0$) is essentially self-adjoint.
\end{itemize}
We also note that the metric $g_{\a}$ for $\a\in\re\setminus \{0\}$ is not geodesically complete.

$(3)$ The problem on the essential self-adjointness of the Clifton-Pohl torus is proposed in \cite{CB}. The author does not know whether the Clifton-Pohl torus satisfies the assumption of \cite[Theorem 12.2]{CB}, which is a criterion for the non-essential self-adjointness on the two-dimensional torus.

\end{rem}

\subsection{Results on the spectral properties}

The spectral properties of $0$-th order operators are studied in \cite{C,Tao} motivated by the studies on forced waves (\cite{CS,DZ2}). In this paper, we deal with the operators of general orders. It is well-known that the spectrum of an elliptic operator with order $m>0$ is discrete on a compact manifold. Therefore, we concentrate on the case where $p$ is not elliptic.

\begin{thm}\label{spthm}
Let $P$ and $V$ be as in Theorem \ref{mainess}. Assume that $p$ is not elliptic.

\noindent $(i)$ Suppose $0< m\leq 1$. We denote the unique self-adjoint extension of $P+V|_{C^{\infty}(\tor)}$ by the same symbol $P+V$. Then the spectrum of $P+V$ is absolutely continuous except discrete eigenvalues.

\noindent$(ii)$ Suppose $m>1$. Then the spectrum of each self-adjoint extension of $P+V|_{C^{\infty}(\tor)}$  is discrete and consists of eigenvalues with finite multiplicity.

\end{thm}

The proof of Theorem \ref{spthm} $(i)$ and $(ii)$ are given in Subsections \ref{subsecMourre} and \ref{subsecdiscspetc} respectively.

\begin{rem}
The spectral properties of the real principal type operators are similar to those of the repulsive Schr\"odinger operator
\begin{align*}
-\Delta-(1+|x|^2)^{\a}\quad \text{on}\quad \re^n\quad \text{for}\quad \a>0.
\end{align*}
It is well-known that the repulsive Schr\"odinger operator is essentially self-adjoint on $C_c^{\infty}(\re^n)$ if and only if $\a\leq 1$. In \cite{BCHM}, for $\a\leq 1$, the authors prove that the spectrum of its unique self-adjoint extension is absolutely continuous. Moreover, in \cite[Corollary 1.5]{T}, it is shown that the spectrum of each self-adjoint extension of the repulsive Schr\"odinger operator is discrete if $\a>1$ and $n=1$.
 \end{rem}


\subsection{Relation to scattering theory}

The idea of the method used in the paper is based on the one in scattering theory.
In this subsection, we briefly explain how classical/quantum incompleteness on a closed manifold $M$ is related to scattering theory. See also \cite[Introduction]{W} which describes the relationship between the microlocal scattering and the traditional one. In earlier works in \cite{NT,NT2,NT3,V2}, the idea of scattering theory has already been used for proving essential self-adjointness of the wave operator on non-compact manifolds.

We recall that all incomplete solutions for ODE on $\re^n$ must escape to infinity at a finite time.
Similarly, if the Hamiltonian vector field $H_p$ is incomplete in $T^*M$, all incomplete trajectories reach the fiber infinity at a finite time since $M$ is closed (compact and boundaryless). 
Consequently, to study incomplete classical trajectories, we must analyze trajectories escaping (``scattering'') to infinity, which is the main interest in classical scattering theory. In the quantum scattering theory, the asymptotic behavior of scattering solutions (generalized eigenfunctions) is analyzed. If one considers it in $T^*M$ with closed manifold $M$, it corresponds to the study of the regularity theory of eigenfunctions since the analysis of the asymptotic behavior of functions at fiber infinity is essentially that of the regularity of these functions in microlocal analysis. On the other hand, the regularity of generalized functions is closely related to the essential self-adjointness of the operator as is explained in Section \ref{sectionSAreg}. That is how quantum scattering theory can be related to self-adjointness theory.

To digest the above intuitive explanation into more practical one, we briefly review a consequence of the stationary scattering theory (\cite[\S 3]{DZ} or \cite{M}) for $-\Delta+V(x)-\l$ with $V\in C_c^{\infty}(\re^n;\re)$ and $\l>0$.
 The outgoing/incoming spherical waves of the form $\frac{1}{|x|^{\frac{n-1}{2}}}a_{\pm}\left(\frac{x}{|x|}\right)e^{\pm i\sqrt{\l}|x|}$ with $a_{\pm}\in L^2(\mathbb{S}^{n-1})$ have an important role in this theory. In fact, all approximate solutions to $(-\Delta+V(x)-\l)u(x)\in\mathcal{S}(\re^n)$ with the minimal growth rate, which represent scattering states can be written as 
\begin{align}\label{scatexp}
u(x)=\frac{1}{|x|^{\frac{n-1}{2}}}a_{+}\left(\frac{x}{|x|}\right)e^{ i\sqrt{\l}|x|}+\frac{1}{|x|^{\frac{n-1}{2}}}a_{-}\left(\frac{x}{|x|}\right)e^{-i\sqrt{\l}|x|}+O\left(\frac{1}{|x|^{\frac{n+1}{2}}}\right)\,\, |x|\to \infty.
\end{align}
for some $a_{\pm}\in L^2(\mathbb{S}^{n-1})$. Moreover, if $u,v$ have such asymptotic expansions, then the boundary pairing formula (\cite[Theorem 3.39]{DZ} and \cite[Proposition 13]{M}) holds:
\begin{align}
&((-\Delta+V-\l)u,v)_{L^2(\re^n)}-(u,(-\Delta+V-\l)v)_{L^2(\re^n)}\nonumber\\
&=2i\l\left((a_-,b_-)_{L^2(\mathbb{S}^{n-1})}-(a_+,b_+)_{L^2(\mathbb{S}^{n-1})}\right)\label{boundrypair}
\end{align}
where $b_{\pm}$ are the ones in the expansion \eqref{scatexp} for $v$ and the left hand side is well-defined (albeit $u,v\notin L^2(\re^n)$). The boundary term in the right hand side of \eqref{boundrypair}, which  breaks the ``symmetric property'' of $-\Delta-\l$ since $((-\Delta-\l)u,v)_{L^2(\re^n)}\neq (u,(-\Delta-\l)v)_{L^2(\re^n)}$ unless the boundary term is zero. In other words, the boundary paring formulas shows how scattering particles give birth to boundary terms and hence break ``symmetric property'' of the operator.
Fortunately, the function $u$ satisfying \eqref{scatexp} does not belong to $L^2(\re^n)$ (unless $a_{\pm}=0$) and it does not contradict to the fact that $-\Delta$ is essentially self-adjoint.


For comparison, we consider a simple example in the boundary value problem for $-\pa_x^2$ on $[0,1]$. The classical trajectories of this operator are just straight lines and reach at the boundary of $[0,1]$ at a finite time. Hence, this system is classically incomplete. Moreover, Green's theorem (the integration by parts) yields
\begin{align}\label{Green}
(-\pa_x^2u,v)_{L^2}-(u,-\pa_x^2v)_{L^2}=[u(x)v'(x)-u'(x)v'(x)]_{x=0}^{x=1},\quad u,v\in H^2([0,1]).
\end{align}
The identity \eqref{Green} implies that $-\pa_x^2$ is not symmetric on its maximal domain and hence $-\pa_x^2$ is not essentially self-adjoint, that is, is quantum incomplete (this is also checked by the operator $-\pa_x^2$ has two different self-adjoint extensions, say the Dirichlet Laplacian and the Neumann Laplacian).
We emphasize that the Green formula \eqref{Green} is very similar to the boundary pairing formula \eqref{boundrypair} and the difference between them is whether the functions $u,v$ belong to $L^2$ or not.

The key observation of the proof of Theorem \ref{mainess} $(ii)\Rightarrow (iii)$ is that the construction approximate $L^2$-eigenfunctions satisfying a formula like \eqref{scatexp} is sufficient to prove non-essential self-adjointness of the operator.
Actually, it is a bit cumbersome to deduce the boundary paring formula. However, the construction of exact eigenfunctions as explained in the next subsection still works for non-real $z$ in our setting (when $m>1$), which directly implies the non essential self-adjointness of the operator. Therefore, we can skip the proof of the boundary pairing formula and proceed to simplify it. For the proof of the boundary paring formula in the $0$-th order case on the two-dimensional torus, see \cite[\S 6]{W}.

\subsection{Idea of the proof of Theorem \ref{mainess}}\label{subsecoutline}

Here we give the idea of the proof of Theorem \ref{mainess} $(ii)\Rightarrow (iii)$ for $V=0$ based on the method developed in \cite{T}. This is the most involved part of the paper. To do this, it suffices to prove $P$ is not essentially self-adjoint on $C^{\infty}(\mathbb{T})$ assuming $m>1$ and $p$ is not elliptic (and satisfies the real principal type condition).
As is known (\cite[Corollary after Theorem VIII. 3]{RS}), we only need to construct non-trivial $L^2$-eigenfunctions $u_{\pm}$ satisfying $(P-z_{\pm})u_{\pm}=0$ with $\pm \Im z_{\pm}>0$. We write $z=z_{\pm}$.

We would like to construct eigenfunctions which satisfy an analogue of \eqref{scatexp}.
These construction is divided into three parts:
\begin{itemize}
\item The construction of the outgoing resolvent $R_+(z)\in B(C^{\infty}(\mathbb{T}), L^2(\mathbb{T}))$ for generic $z\in\mathbb{C}$ via anisotropic Sobolev spaces motivated by the methods in \cite{FS} and \cite{V1}. 
Roughly speaking, the outgoing resolvent is defined to satisfy the following property: The image $R_+(z)f$ has better regularity at the incoming Lagrangian submanifold $\L_{\mathrm{inc}}$ (defined in Subsection \ref{escapesubst} and it is also called the radial source, see \cite[Definition E.50]{DZ}) and worse regularity at the outgoing Lagrangian submanifold (radial sink) $\L_{\mathrm{out}}$. In other words, it satisfies a microlocal analogue of Sommerfeld's radiation condition, see \cite[$(11.1)$]{Me} for scattering theory;

\item The construction of a non-trivial incoming WKB ansatz $u_{\mathrm{inc},z}\in L^2(\mathbb{T})$ such that $(P-z)u_{\mathrm{inc},z}$ is negligible, that is, a quasimode (an approximate eigenfunction) of $P-z$. This is a Lagrangian distribution associated with the incoming Lagrangian submanifolds  $\L_{\mathrm{inc}}$;

\item Now $u=u_{\mathrm{inc},z}-R_+(z)(P-z)u_{\mathrm{inc},z}\in L^2(\mathbb{T})$ satisfies $(P-z)u=0$ and more importantly $u\neq 0$ due to the wave front condition of $u_{\mathrm{inc},z}$. In fact, $u_{\mathrm{inc},z}$ has non-trivial wavefront in the incoming Lagrangian submanifolds $\L_{\mathrm{inc}}$ whereas the image of $R_+(z)$ has trivial wavefront there (actually a weaker statement is sufficient for our purpose, see the proof of Proposition \ref{eigenconst}).

\end{itemize}
This is the well-known construction of perturbed spherical waves in scattering theory (\cite{GY} and \cite[Section 12]{Me}) for $P=-\Delta+V(x)-\l$ on $\re^n$ with $V\in C_c^{\infty}(\re^n;\re)$ and $\l>0$ although there are few differences between this case and our setting;
\begin{itemize}
\item The outgoing resolvent $R_+(1)$ for $-\Delta-1$ just maps from $\mathcal{S}(\re^n)$ to the weighted $L^2$-space $L^{2,-\frac{1}{2}-0}(\re^n)(=\bigcap_{\e>0}(1+|x|)^{-\frac{1}{2}-\e}L^2(\re^n)) $ and the image is not contained in $L^2(\re^n)$. Moreover, the (essential) self-adjointness of $-\Delta$ is used to construct $R_+(1)$ since it is usually defined by the limit of the $L^2$ resolvent: $R_+(1):=\lim_{\e\searrow 0}(-\Delta-1-i\e)^{-1}$;

\item Non-trivial incoming WKB ansatz does not belong to $L^2(\re^n)$ since it is of the form $|x|^{-\frac{n-1}{2}}a_-\left(\frac{x}{|x|}\right) e^{-i|x|}$ for $|x|\gg 1$.

\end{itemize}
Indeed, we can show that the range of our outgoing resolvent $R_+(z)$ is included in $H^{\frac{m-1}{2}-\e}(\mathbb{T})) $ for some $\e>0$ but is not in $H^{\frac{m-1}{2}}(\mathbb{T})$ and that our incoming ansatz $u_{so,z}$ belongs to $H^{\frac{m-1}{2}-\e}(\mathbb{T})\setminus H^{\frac{m-1}{2}}(\mathbb{T})$. 
This is why the range of our outgoing resolvent is contained in $L^2(\mathbb{T})$ when the order $m$ of the operator $P$ is greater than $1$ since $H^{\frac{m-1}{2}-0}(\mathbb{T})\subset L^2(\mathbb{T})$ if only if $m>1$.
Furthermore, we construct the outgoing resolvent by a microlocal method. Its merit is that the essential self-adjointness of the operator is not needed for its construction. This is very important since our purpose here is to prove that $P$ is not essentially self-adjoint.

\subsection{Possible generalization to higher dimensional cases}

The method explained in the last subsection is robust enough perhaps to generalize the result on non-essential self-adjointness of operators to higher dimensional cases. In fact, Theorem \ref{Fredcor} shows that the outgoing resolvent can be constructed in more general situations when one assumes the existence of an escaping function. On the other hand, the construction of the WKB ansatz is more involved. First, the incoming set (the radial source) is not always smooth like in the case of Anosov flows. Therefore, it seems difficult to construct a WKB ansatz (a Lagrangian distribution) associated with this set. Second, even though this set is smooth and Lagrangian, it is not trivial whether the transport equation is solvable. Its solvability is needed to obtain a better approximation of solutions (see \cite[Theorem 25.2.4]{Ho} and \cite[\S 10.2]{Z}). The existence of a solution of the transport equation is something to do with that of a flow invariant (Maslov) half-density (see \cite[\S 10.2, Remark]{Z}). The natural question is whether a conic Lagrangian submanifold which is a radial source or sink admits a flow invariant (Maslov) half-density.

The author believes that the proof in this paper would work in the situation considered in \cite{CS}, \cite{DZ2}, and \cite{W} since the structure of the incoming Lagrangian submanifold is well understood and an invariant half-density exists there (\cite[Lemma 2.5]{DZ}) albeit the case $m=0$ (but the construction of an incoming WKB state should work for $m>1$ in the same way as \cite[\S 4]{W}). Here, we do not discuss it more and the generalizations leave for future work.

\subsection{Notation}

We fix some notations. The function space $\mathcal{D}'(M)$ and $H^s=H^s(M)$ denote the set of all distributions and the Sobolev spaces of order $s$ on a closed manifold $M$. We write $H^{k-0}=\bigcap_{\e>0}H^{k-\e}$. For Banach spaces $X,Y$, $B(X,Y)$ denotes the set of all linear bounded operators form $X$ to $Y$. For a Banach space $X$, we denote the norm of $X$ by $\|\cdot\|_{X}$. If $X$ is a Hilbert space, we write its inner metric of $X$ by $(\cdot, \cdot)_{X}$, where $(\cdot, \cdot)_{X}$ is linear with respect to the right variable. We also denote $\|\cdot\|=\|\cdot\|_{L^2}$ and $(\cdot,\cdot)=(\cdot, \cdot)_{L^2}$. For linear operators $A,B$, we write $[A,B]=\mathrm{ad}_AB=AB-BA$.

\noindent
\textbf{Acknowledgment.} 
This work was partially supported by JSPS Research Fellowship for Young Scientists, KAKENHI Grant Number 17J04478, 20J00221, 23K13004 and the program FMSP at the Graduate School of Mathematics Sciences, the University of Tokyo. The author would like to thank Shu Nakamura for helpful discussions. The author also would like to thank Kenichi Ito for encouraging to write this paper and to Keita Mikami for letting me know the paper \cite{CB}. The author is grateful to Yves Colin de Verdi\'ere and Maxim Braverman for their interest on this work.

\section{Pseudodifferential operators}\label{sectionmic}

In this section, we recall some terminology in microlocal analysis on a closed (compact and boundaryless) manifold $M$ of dimension $n$. 
We say that $\c$ is a chart on $M$ if $\c:U\to V$ is a diffoemorphism, where $U\subset M$ and $V\subset \re^n$ are open sets. Fix a smooth (Riemannian) metric $g$ on $T^*M$ and write
\begin{align*}
\jap{\x}:=(1+|\x|_g^2)^{\frac{1}{2}},\quad |\x|_g^2=\sum_{j,k=1}^ng_{jk}(x)\x_j\x_k.
\end{align*}
On the torus, we use the flat Riemmanian metric there and $\jap{\x}=(1+|\x|^2)^{\frac{1}{2}}$, see \eqref{flatjap}.
We denote the zero section of $T^*M$ by $0$.

\subsection{Pseudodifferential operators}

We introduce the concept of the pseudodifferential operators following \cite[\S Appendix E]{DZ}, \cite[\S18.1]{Ho} or \cite[\S 14.1]{Z}.

\subsection*{Symbol classes}
We denote by the Kohn-Nirenberg symbol classes by $S^k$ and the homogeneous symbol class by $S_{\mathrm{hom}}^k$: For $k\in \re$, we define
\begin{align*}
S^k:=&\{a \in C^{\infty}(T^*M)\mid |\pa_{x}^{\a}\pa_{\x}^{\b}a(x,\x)|\leq C_{\a\b}\jap{\x}^{k-|\b|}\},\\
S_{\mathrm{hom}}^k:=&\{a \in C^{\infty}(T^*M)\mid a(x,\l\x)=\l^m a(x,\x)\,\, \text{for}\,\, (x,\x)\in T^*M,\,\, |\x|_g\geq 1,\,\, \l\geq 1\}.
\end{align*}

\subsection*{Pseudodifferential operators and principal/subprincipal symbols}
We denote $\Psi^{-\infty}$ by the set of all linear continuous maps whose integral kernels are smooth on $M\times M$ (\cite[Definition E.11]{DZ}).

\begin{defn}\label{symboldef}
\noindent$(i)$
For $k\in \re$, we define the class of pseudodiffferential operators $\Psi^k$ as follows: A linear continuous map $A:C^{\infty}(M)\to \mathcal{D}'(M)$ lies in $\Psi^k$ if both of the following hold:
\begin{itemize}
\item For each $\f,\g\in C^{\infty}(M)$ with $\supp \f\cap \supp \g=\emptyset$, we have $\f A\g\in \Psi^{-\infty}$;

\item For each chart $\c:U\to V\subset \re^n$, there exists $a\in S^k$ (defined on $T^*\re^n$) such that
\begin{align*}
(\c^{-1})^* A \c^*u(x)=a_{\c}(x,D_x)u(x)\quad \forall u\in C_c^{\infty}(V)
\end{align*}
where $a_{\c}(x,D_x)u(x)=(2\pi)^{-n}\int_{\re^n}\int_{\re^n}a_{\c}(x,\x)e^{i(x-y)\cdot \x}u(y)dyd\x$ is the standard pseudodifferential operator with symbol $a$. We say that $a_{\c}$ is the total symbol of $A$ in the chart $\c$.

\end{itemize}

\noindent$(ii)$ We say that $A\in \Psi^k$ belongs to the class $\Psi^k_{\mathrm{phg}}$  of polyhomogeneous pseudodifferential operators if for each chart $\c$, the total symbol $a$ of $A$ there has asymptotic expansion of the form $a_{\c}\sim \sum_{j=0}^{\infty}a_{m-j,\c}$ with $a_{m-j,\c}$ homogeneous of degree $m-j$.

\noindent$(iii)$ Let $A\in \Psi^k_{\mathrm{phg}}$. From \cite[$(18.1.30)$]{Ho}, $a_{m,\c}$ is invariantly defined, that is there exists $\s(A)\in S^k_{\mathrm{hom}}$ such that $\k_{\c}^*\s(A)=a_{m,\c}$, where the mapping $\k_{\c}:T^*V\to T^*U \subset T^*M$ is induced from the transpose map of the exterior derivative $d\c:TU\to TV$. We say that $\s(A)$ is the principal symbol of $A\in \Psi^k_{\mathrm{phg}}$. Moreover, the subprincipal symbol of $A\in \Psi^k_{\mathrm{phg}}$ in the chart $\c$ is defined by
\begin{align*}
\s_{\mathrm{sub}}(A)=a_{m-1,\c}+\frac{i}{2}\sum_{j=1}^n\frac{\pa^2a_{m,\c}}{\pa x_j\pa \x_j}.
\end{align*}

\end{defn}

\begin{rem}
Unfortunately, the subprincipal symbol $\s_{\mathrm{sub}}(A)$ is not invariantly defined. In fact, \cite[$(18.1.33)$]{Ho} shows that if $\c'$ is another chart, then
\begin{align*}
\left(a_{m-1,\c'}+\frac{i}{2}\sum_{j=1}^n\frac{\pa^2a_{m,\c'}}{\pa x_j\pa \x_j}\right)-\left(a_{m-1,\c}+\frac{i}{2}\sum_{j=1}^n\frac{\pa^2a_{m,\c}}{\pa x_j\pa \x_j} \right)=-\frac{1}{2}\sum_{j=1}^nJ^{-1}\pa_{x_j}J(\pa_{\x_j}a_{m,\c})
\end{align*}
where $J(x)=\det(\pa (\c'\circ \c^{-1})(x))$.
However, it is known that $\Psi^k_{\mathrm{phg}}$ acts on half-densities instead of functions, then $\s_{\mathrm{sub}}(A)$ defined above is invariantly defined (\cite[Theorem 18.1.33]{Ho}).
\end{rem}


\subsection{Quantization}\label{subsecquant}

Fix a finite collection of charts $\c_j:U_j\to V_j$ ($j=1,\hdots,J$) such that $\cup_{j=1}^JU_j=M$ and take a quadratic partition of unity $\{\chi_j\}_{j=1}^J\subset C^{\infty}(M)$, that is $\sum_{j=1}^J\chi_j^2=1$ everywhere and $\supp \chi_j\subset U_j$. For $k\in \re$ and $a\in S^k$, we define the quantization of $a$ by
\begin{align*}
\Op(a):=\sum_{j=1}^J \chi_j\c_j^* a_j(x,D_x)(\c_j^{-1})^*  \chi_j
\end{align*}
where $a_j(x,\x)=a(\c_j^{-1}(x), ({}^t\pa \c_j)(\c_j^{-1}(x))\x)$ for $(x,\x)\in T^*V_j$.

\subsection{Fundamental properties of pseudodifferential operators}\label{subsecsymbolcl}

Fix a smooth density $d\m$ on $M$. In this subsection, we denote the $L^2$-space induced from $d\m$ by $L^2(d\m)=L^2(M;d\m)$ in order to stress that density, where $(u,w)_{L^2(d\m)}=\int_M\overline{u(x)}w(x)d\m(x)$.

Moreover, we define the Hamilton vector field $H_a$ of $a\in C^{\infty}(T^*M )$ on $T^*M$ by 
\begin{align}\label{Hamvecdef}
H_a=\pa_{\x}a(x,\x)\cdot\pa_x-\pa_xa(x,\x)\cdot \pa_{\x}
\end{align}
in each local coordinate. For $a_j\in S^{k_j}$ for $j=1,2$, $H_{a_1}a_2\in S^{k_1+k_2-1}$.

\begin{lem}\label{PDOfund}
Let $k\in \re$ and $k_j\in \re$ for $j=1,2$.

\noindent$(i)$ If $k_1\leq k_2$, then $\Psi^{k_1}\subset \Psi^{k_2}$.

\noindent$(ii)$ $A_j\in \Psi^{k_i}$, then $A_1A_2\in \Psi^{k_1+k_2}$.

\noindent$(iii)$ If $A\in \Psi^k$, then $A\in B(H^{k+\ell},H^{\ell})$  for each $\ell \in\re$. If $k<0$, then $A$ is a compact operator on $L^2$.

\noindent$(iv)$
Let $a_j\in S^{k_j}$ for $j=1,2$. Then
\begin{align*}
[\Op(a_1), i\Op(a_2)]-\Op(H_{a_1}a_2)\in \Psi^{k_1+k_2-2},
\end{align*}
where we note $H_{a_1}a_2\in S^{k_1+k_2-1}$. 

\noindent$(v)$$($Sharp G\aa rding inequality$)$ Let $a\in S^{2k+1}$ and $A:=\Op(a)\in \Psi^{2k+1}$ with $k\in \re$ and $\Re a\geq 0$. Then there exists $C>0$ such that
\begin{align*}
\Re(u, Au)_{L^2}\geq -C\|u\|_{H^{k}}^2,\quad u\in C^{\infty}(M).
\end{align*}

\noindent$(vi)$ Let $A\in \Psi^k$. Assume that $u\in H^s$ and $Au\in H^{s-k+1}$. Then there exists a sequence $u_j\in C^{\infty}$ such that
\begin{align*}
\|u_j-u\|_{H^s}\to 0,\quad \|Au_j-Au\|_{H^{s-k+1}}\to 0.
\end{align*}

\end{lem}

\begin{proof}

These properties are proved by the standard arguments, for example, see  \cite[Appendix E]{DZ}, \cite[\S 18.1]{Ho} or \cite[\S 14]{Z}. The property $(vi)$ is stated in \cite[Lemma E.45]{DZ}.

\end{proof}

\begin{defn}\label{pridef}
Let $m\geq 0$ and $a\in S^m$.

\noindent$(i)$ We say that $a$ is elliptic if there is $C_1,C_2>0$ such that
\begin{align*}
|a(x,\x)|\geq C_1|\x|_g^m-C_2
\end{align*}

\noindent$(ii)$
We say that $a$ is real principal type if $a$ is real-valued and
\begin{align*}
a(x,\x)=0\Rightarrow da(x,\x)\neq 0\quad \text{for}\quad |\x|_g\geq 1.
\end{align*}

\end{defn}

\subsection{Wavefront set}

We briefly recall the notion of the (Sobolev) wavefront sets.

\begin{defn}
Let $u\in \mathcal{D}'(M)$ and $s\in \re$. We define the Sobolev wavefront set $\mathrm{WF}^s(u)$ of $u$ as follows: For $(x_0,\x_0)\in T^*M\setminus 0$, we say that $(x_0,\x_0)\notin \mathrm{WF}^s(u)$ if and only if there exists $a\in S^0$ and a conic neighborhood $U$ of $(x_0,\x_0)$ such that $a=1$ in $U\cap \{|\x|_g\geq 1\}$ and $\Op(a)u\in H^s(M)$.
We also define $\mathrm{WF}(u):=\cup_{s\in \re}\mathrm{WF}(u)$.

\end{defn}

\subsection{Lagrangian distributions}

In this subsection, we briefly discuss Lagrangian distributions. Here we do not get into the general theory of Lagrangian distributions in order to avoid using the half-density bundle and the Maslov bundle and concentrate on a simple conic Lagrangian submanifold
\begin{align}\label{simpleLag}
\L_{x_0,\pm}:=\{(x_0,\x)\in T^*\tor\mid \pm \x>0\},\quad x_0\in \tor.
\end{align}
In this paper, we use standard coordinates of $\tor$ in the following sense just for simplicity.
\begin{defn}\label{stdcoord}
We say that a coordinate of $\tor$ is standard if it is induced from the natural quotient map $\re\to \re/2\pi\mathbb{Z}=\tor$.

\end{defn}

Throughout of this paper, we use the flat Riemmanian metric on the torus and hence 
\begin{align}\label{flatjap}
\jap{\x}=(1+|\x|^2)^{\frac{1}{2}}
\end{align}
on a standard coordinate.
For $\frac{1}{2}<\d\leq 1$, we define
\begin{align*}
S_{\d,\pm}^{\ell}(\re):=\{b\in C^{\infty}(\re)\mid \forall\a\in\mathbb{N},\,\, \sup_{\x\in \re}\jap{\x}^{-\ell+\d|\a|}|\pa_{\x}^{\a}b(\x)|<\infty ,\quad  \supp b\subset \pm(0,\infty) \},
\end{align*}
where we denote $+(0,\infty))=(0,\infty)$ and $-(0,\infty)=(-\infty,0)$.

We define the space $I^{\ell}_{\d}(\tor,\L_{x_0,\pm})$ to be the set of all  $u\in\mathcal{D}'(\tor)$ such that $u\in C^{\infty}(\tor\setminus \{x_0\})$ and there exist $b\in S_{\d,\pm}^{\ell-\frac{1}{4}}(\re)$ and a smooth function $r$ near $x_0$ such that
\begin{align}\label{Lagdisloc}
u(x)=\int_{\re}b(\x)e^{i(x-x_0)\x}d\x+r
\end{align}
in a standard coordinate near $x_0$. We call an element of $I^{\ell}_{\d}(\tor,\L_{x_0,\pm})$ a Lagrangian distribution of order $(\ell,\d)\in \re\times (\frac{1}{2},1]$ associated with $\L_{x_0,\pm}$. Usually, a special case $\d=1$ is considered (\cite[Definition 25.1.1, Proposition 25.1.3]{Ho}) and most of the general theory in \cite[\S 25]{Ho} can be applied to our space $I^{\ell}_{\d}(\tor,\L_{x_0,\pm})$. 

\begin{lem}\label{Laglem}
Let $x_0\in \tor$, $\ell\in \re$, $0<\d\leq 1$ and $u\in I^{\ell}_{\d}(\tor,\L_{x_0,\pm})$. Let $A\in \Psi^k_{\mathrm{phg}}$ and set $a=\s(A)$.
Then we have the following:

\noindent$(i)$ $u\in H^{-\frac{1}{4}-\ell-0}(\tor)$ and $\mathrm{WF}(u)\subset \L_{x_0,\pm}$.

\noindent$(ii)$ We have $Au\in I^{k+\ell}_{\d}(\tor,\L_{x_0,\pm})$. Moreover, if we write $u$ as in \eqref{Lagdisloc}, then there exists $u'\in I^{k+\ell-\d}_{\d}(\tor,\L_{x_0,\pm})$ 
\begin{align*}
Au(x)=\int_{\re} a(x_0,\x)b(\x)e^{i(x-x_0)\x}d\x   +u'
\end{align*}
in the same coordinate as \eqref{Lagdisloc}.

\noindent$(iii)$ We assume $a(x_0,\x)=0$ for all $\x\in \supp b$, where $b$ is as in \eqref{Lagdisloc}. Then there exists $u''\in I^{k+\ell-2\d}_{\d}(\tor,\L_{x_0,\pm})$ such that
\begin{align*}
Au(x)=\int_{\re} \left(i^{-1}(\mathcal{L}_{H_a}b)(\x) +\s_{\mathrm{sub}}(A)(x_0,\x) b(\x)\right)e^{i(x-x_0)\x}d\x   +u''
\end{align*}
in the same coordinate as \eqref{Lagdisloc}, where 
\begin{align*}
(\mathcal{L}_{H_a}b)(\x)=-(\pa_xa)(x_0,\x)\pa_{\x}b(\x)-\frac{1}{2}(\pa_x\pa_{\x}a)(x_0,\x)b(\x)
\end{align*}
and $\s_{\mathrm{sub}}(A)$ is the subprincipal symbol defined in Definition \ref{symboldef}.

\end{lem}

\begin{rem}
The notation $\mathcal{L}_{H_a}$ is slightly different from the one in \cite[Theorem 25.2.4]{Ho}. Indeed, $\mathcal{L}_{H_a}$ is the Lie-derivative of $H_a$ acting on half-densities $\tilde{b}(\x)=b(\x)|d\x|^{\frac{1}{2}}$ on $\L_{x_0,\pm}$ there. Here, we identify $b$ with $\tilde{b}$ in the standard coordinate of the torus.
\end{rem}

For the proof, see \cite[Theorem 25.2.4]{Ho}.

\section{Some criterions of the essential self-adjointness}\label{sectionSAreg}

In this section, we give some criterions for the essential self-adjointness on a more general closed manifold $M$. We fix a smooth density $d\m$ on $M$ and denote the $L^2$-space with respect to $d\m$ by $L^2=L^2(M;d\m)$.

Now let $m\geq 0$ and $\tilde{P}\in \Psi^m$ which is symmetric on $L^2$ with its domain $C^{\infty}(M)$: $(u,\tilde{P}w)_{L^2}=(\tilde{P}u,w)_{L^2}$ for all $u,w\in C^{\infty}(M)$. We define the maximal domain of $\tilde{P}$ by
\begin{align}\label{Dmaxdef}
D_{\mathrm{max}}(\tilde{P})=\{u\in L^2\mid \tilde{P}u\in L^2\},
\end{align}
where we consider $\tilde{P}u$ in the distributional sense. The maximal domain is just the domain of the adjoint operator of $\tilde{P}|_{C^{\infty}(M)}$.

\subsection{Subellipticity}

Here we show that the essential self-adjointness of $\tilde{P}$ can be deduced from a kind of subellipticity as follows:

\begin{prop}\label{esscri1}
If $D_{\mathrm{max}}(\tilde{P})\subset H^{\frac{m-1}{2}}$ holds, then $\tilde{P}$ is essentially self-adjoint on $C^{\infty}(M)$. In particular, if $m\leq 1$ or $\tilde{P}$ is elliptic, then $\tilde{P}$ is essentially self-adjoint on $C^{\infty}(M)$.

\end{prop}

\begin{rem}
$(1)$
It turns out from our main result (Theorem \ref{mainess}) that the exponent $\frac{m-1}{2}$ is critical. In fact, Proposition \ref{radsinkpropes} implies that $D_{\mathrm{max}}(\tilde{P})\subset H^{\frac{m-1}{2}-0}$ with $\tilde{P}=P+V$ (under the real principal type condition). On the other hand, this operator is not essentially self-adjoint on $C^{\infty}(M)$ when $m>1$ and when $p$ is real principal type but is not elliptic due to Theorem \ref{mainess}.

$(2)$ The converse is not true. We consider $\tilde{P}=\pa_x^2-\pa_y^2$ on $M=\tor^2_{(x,y)}$. Then $\tilde{P}$ is essential self-adjoint on $C^{\infty}(\tor^2)$ but there is $u\in D_{\mathrm{max}}(\tilde{P})$ such that $u\notin H^{\e}(\tor^2)$ for all $\e>0$. In fact, the essential self-adjointness is easy to prove by using the Fourier series. Moreover, $u(x,y):=\sum_{n=1}^{\infty}\frac{1}{n^{\frac{1}{2}}(1+\log n)}e^{i(x-y)n}\in L^2(\tor^2)$ satisfies $\tilde{P}u=0$ and hence $u\in D_{\mathrm{max}}(\tilde{P})$. However, $u\notin H^{\e}(\tor)$ for all $\e>0$.

\end{rem}

To prove this proposition, we need some lemmas.

\begin{lem}\label{esscrilem1}
If $\tilde{P}$ is symmetric on $D_{\mathrm{max}}(\tilde{P})$ in the sense that $(u,\tilde{P}w)_{L^2}=(\tilde{P}u,w)_{L^2}$ for all $u,w\in D_{\mathrm{max}}(\tilde{P})$, then $\tilde{P}$ is essentially self-adjoint on $C^{\infty}(M)$.
\end{lem}

\begin{proof}
We recall that $D_{\mathrm{max}}(\tilde{P})$ is the domain of the adjoint operator of $\tilde{P}|_{C^{\infty}(M)}$: $D_{\mathrm{max}}(\tilde{P})=D((\tilde{P}|_{C^{\infty}(M)})^*)$. This means that the adjoint operator $(\tilde{P}|_{C^{\infty}(M)})^*$ is symmetric and hence $D((\tilde{P}|_{C^{\infty}(M)})^*)\subset D((\tilde{P}|_{C^{\infty}(M)})^{**})$. On the other hand, $(\tilde{P}|_{C^{\infty}(M)})^{**}$ is equal to the operator closure of $\tilde{P}|_{C^{\infty}(M)}$. Hence, $(\tilde{P}|_{C^{\infty}(M)})^*$ coincides with the operator closure of $\tilde{P}|_{C^{\infty}(M)}$, which implies the essential self-adjointness.

\end{proof}

Next, we recall a general result about pseudodifferential operators. The following lemma can be proved by an approximate argument.

\begin{lem}\cite[Lemma E.46]{DZ}
Let $s\in \re$. If $u\in H^s$ and $w\in H^{-s+m-1}$ satisfies $\tilde{P}u\in H^{s-m+1}$ and $\tilde{P}w\in H^{-s}$, then
\begin{align}\label{generalsymmet}
(\tilde{P}u,w)_{L^2}=(u,\tilde{P}w)_{L^2}.
\end{align}

\end{lem}

\begin{cor}\label{esscricor1}
If $u,w\in D_{\mathrm{max}}(\tilde{P})$ satisfies $u,w\in H^{\frac{m-1}{2}}$, then \eqref{generalsymmet} holds.
\end{cor}

\begin{proof}
We just take $s=\frac{m-1}{2}$ in the last lemma.
\end{proof}

\begin{proof}[Proof of Proposition \ref{esscri1}]
Proposition \ref{esscri1} directly follows from Lemma \ref{esscrilem1} and Corollary \ref{esscricor1}.
\end{proof}

\subsection{Non-real eigenfunctions}

The next proposition claims that the essential self-adjointness is equivalent to the non-existence of non-real $L^2$-eigenfunctions.

\begin{prop}\label{esscrieigen}
$\tilde{P}$ is essentially self-adjoint on $C^{\infty}(M)$ if and only if there exists $\m_{+},\m_-\in\mathbb{C}$ such that $\pm\Im \m_{\pm}>0$ and
\begin{align}\label{distriker}
u\in L^2,\,\, (\tilde{P}+ \m_{\s})u=0\Rightarrow u=0
\end{align}
for each $\s\in \{\pm\}$. Here $(\tilde{P}+ \m_{\s})u$ is defined in the distributional sense.

\end{prop}

\begin{proof}

We recall again that $D_{\mathrm{max}}(\tilde{P})$ is the domain of the adjoint operator of $\tilde{P}|_{C^{\infty}(M)}$. Then the condition \eqref{distriker} is equivalent to $\Ker((\tilde{P}|_{C^{\infty}(M)})^*+\m_{\s})=\{0\}$. Now the claim follows from  \cite[Corollary after Theorem VIII.3, Theorem X.1]{RS}.

\end{proof}

As a corollary, we show that the assumption of Proposition \ref{esscri1} can be relaxed. Indeed, Proposition \ref{esscri1} directly follows from Corollary \ref{Esscricor} below. However, Proposition \ref{esscri1} reveals the relationship between the subellipticity of $\tilde{P}$ and is important itself. Therefore, another proof is given in the last subsection.
 Although the following corollary is not used in this paper, its analog on non-compact manifolds has an important role in \cite{NT,NT2,NT3}.

\begin{cor}\label{Esscricor}
If there exists $\m_{+},\m_-\in\mathbb{C}$ such that $\pm\Im \m_{\pm}>0$ and $\{u\in L^2\mid (\tilde{P}+\m_{\s})u=0\} \subset H^{\frac{m-1}{2}}$ for each $\s\in \{\pm\}$, then $\tilde{P}$ is essentially self-adjoint on $C^{\infty}(M)$.

\end{cor}

\begin{proof}
Suppose that $u\in L^2$ satisfies $(\tilde{P}+\m_{\s})u=0$. Then $\tilde{P}u=-\m_{\s}u\in L^2$ and hence $u\in D_{\mathrm{max}}(\tilde{P})$.
By Corollary \ref{esscricor1}, we have
\begin{align*}
\m_{\s}\|u\|_{L^2}^2= (u,\tilde{P}u)_{L^2}=(\tilde{P}u,u)_{L^2}=\overline{\m_{\s}}\|u\|_{L^2}^2.
\end{align*}
Since $\Im \m_{\s}\neq 0$, we obtain $u=0$. Now the claim follows from Proposition \ref{esscrieigen}.

\end{proof}

\section{Escaping function and Fredholm property}

Let $M$ be a closed manifold and $d\m$ be a smooth density on $M$.
In this section, we show a Fredholm property for a pseudodifferential operator $\tilde{P}$ assuming the existence of an escaping function. We follow the argument in \cite[\S 3]{FS}, see also \cite[\S 3]{T} and \cite{V1}. We write $L^2=L^2(M;d\m)$ and $H^s=H^s(M)$.

Suppose that $m>1$ and $\tilde{P}\in \Psi^m$ is symmetric on $C^{\infty}(M)$ with respect to the inner metric of $L^2(d\m)$. We assume that a real-valued symbol $\tilde{p}\in S^m$ satisfies
\begin{align*}
\tilde{P}-\Op(\tilde{p})\in \Psi^{m-1}.
\end{align*}
We define the elliptic set of $\tilde{p}$ by
\begin{align*}
\mathrm{ell}_{r,R}(\tilde{p}):=\{(x,\x)\in T^*M\mid |\tilde{p}(x,\x)|\geq r\jap{\x}^m \}\cup \{(x,\x)\in T^*M\mid |\x|_g\leq R\}
\end{align*}
for $r,R>0$.

\subsection{Definition of an escaping function}

\begin{defn}\label{esces}
A real-valued function $k\in S^0$ is called an escaping function of $\tilde{P}$ if there exist $C_0>0$, $r_0>0$ $R_0>0$ such that
\begin{align}\label{escdefest}
H_{\tilde{p}}(k(x,\x)\log\jap{\x})\leq -C_0\jap{\x}^{m-1}+d(x,\x)^2,\quad \,\, (x,\x)\in T^*M,
\end{align}
where $d\in S^{\frac{m-1}{2}}(T^*M)$ is a real-valued symbol satisfying $\supp d\subset \mathrm{ell}_{r_0,R_0}(\tilde{p})$.

\end{defn}

In Lemma \ref{escexist}, we prove the existence of an escaping function for $\tilde{P}=P+V$ on $M=\tor$ under the assumptions of Theorem \ref{mainess}.
An immediate consequence of the existence of an escaping function is a subprincipal estimate as follows:

\begin{lem}\label{escpos}
Suppose that there exists an escaping function $k\in S^0$ of $\tilde{P}$.
There exits $C_1>0$ such that for $u\in C^{\infty}(M)$,
\begin{align*}
-(u, \Op(H_{\tilde{p}}(k\log\jap{\x})u)_{L^2}\geq C_0\|u\|_{H^{\frac{m-1}{2}}}^2-C_1\|Du\|_{L^2}^2-C_1\|u\|_{H^{\frac{m-2}{2}+0}}^2,
\end{align*}
where $D:=\Op(d)$ and $d\in S^{\frac{m-1}{2}}$ is as in Definition \ref{esces}.
\end{lem}

\begin{proof}
We note that $k\log\jap{\x}\in S^{+0}=\cap_{\e>0}S^{\e}$. Then this lemma follows from the sharp G\aa riding inequality (Lemma \ref{PDOfund}).

\end{proof}

\subsection{Fredholm property for the deformed operator}

Throughout of this subsection, we assume the existence of an escaping function $k\in S^0$.

Let $k$ be a symbol constructed in the above subsection.
Take an invertible operator $G_{k}$ as in Lemma \ref{Anistrolem} $(ii)$ such that the principal part of the symbol of $G_k$ is $\jap{\x}^{k(x,\x)}$.
Define 
\begin{align}\label{Ptmdef}
\tilde{P}_{k}:=G_{k}\tilde{P}G_{k}^{-1}.
\end{align}
By Lemma \ref{Anistrolem} $(iv)$ (see \cite[Lemma 3.2]{FS} for the Anosov vector field), we have
\begin{align}\label{Ptmfor}
\tilde{P}_{k}=\tilde{P}+i\Op(H_{\tilde{p}}(k\log\jap{\x}))+\Psi^{m-2+0},
\end{align}
where $\Psi^{m-2+0}=\cap_{\e>0}\Psi^{m-2+\e}$.

We consider the Banach space
\begin{align*}
\tilde{D}_{k}:=\{u\in H^{\frac{m-1}{2}}\mid \tilde{P}_{k}u\in H^{-\frac{m-1}{2}}\}
\end{align*}
equipped with the graph norm of $P_{k}$. The purpose of this subsection is to prove generic invertibility of deformed operators $\tilde{P}_k-z$ as follows.

\begin{thm}\label{Fredcor}
Consider a family of bounded operators
\begin{align}\label{Ptmmap}
\tilde{P}_{k}-z:\tilde{D}_{k}\to H^{-\frac{m-1}{2}}.
\end{align}
for $z\in\mathbb{C}$.
Then it follows that the $(\ref{Ptmmap})$ is an analytic family of Fredholm operators. Moreover, there exists a discrete subset $S\subset \mathbb{C}$ such that the map \eqref{Ptmmap} is invertible for $z\in \mathbb{C}\setminus S$.

\end{thm}

\begin{rem}
From this result, we can define the outgoing resolvent of $\tilde{P}$ as $G_k^{-1}(\tilde{P}_k-z)^{-1} G_k$ between anisotropic Sobolev spaces. Although we assume that $m>1$ here, the proof works even for $m\leq 1$ at least when $\Im z>0$ and the outgoing resolvent is just the $L^2$-resolvent of the unique self-adjoint extension of $\tilde{P}$ (see \cite[\S 4]{FS} for Anosov vector fields).

\end{rem}

\begin{lem}\label{domaindense}
$C^{\infty}(M)$ is dense in $\tilde{D}_k$.
\end{lem}

\begin{proof}

This lemma immediately follows from Lemma \ref{PDOfund} $(vi)$ with $s=\frac{m-1}{2}$ and $A=\tilde{P}_k$.
\end{proof}

For $I\Subset \re$ and $L>0$, we define
\begin{align*}
I_{L}=\{z\in \mathbb{C}\mid \Re z\in I,\,\,  \Im z>-L\}.
\end{align*}
The main results of this subsection are the following Fredholm estimates.

\begin{prop}\label{Fredholmes}
Let $I\Subset \re$ and $L>0$.
Then for any $N>0$, there exists $C>0$ such that
\begin{align}
\|u\|_{H^{\frac{m-1}{2}}}\leq& C\|(\tilde{P}_{k}-z)u\|_{H^{-\frac{m-1}{2}}}+C\|u\|_{H^{-N}}, \label{Frees1}\\
\|u\|_{H^{\frac{m-1}{2}}}\leq& C\|(\tilde{P}_{k}-z)^*u\|_{H^{-\frac{m-1}{2}}}+C\|u\|_{H^{-N}},\label{Frees2}
\end{align}
for $z\in I_L$ and $u\in \tilde{D}_{k}$. Moreover, if $\Im z>>1$, then the term $\|u\|_{H^{-N}}$ in \eqref{Frees1} and \eqref{Frees2} can be removed.
\end{prop}

\begin{rem}
For $m\leq 1$, the above Fredholm estimates hold when $z\in I_0$.
\end{rem}

\begin{proof}[Proof of Proposition \ref{Fredholmes}]
We only deal with $(\ref{Frees1})$.  The inequality $(\ref{Frees2})$ is similarly proved.
By virtue of Lemma \ref{domaindense}, it suffices to prove $(\ref{Frees1})$ for $u\in C^{\infty}(M)$. 

Since $\tilde{P}$ is symmetric, we have $\Im (u,(\tilde{P}-z)u)_{L^2}=-\Im z\|u\|_{L^2}^2$.
By Lemma \ref{escpos} and \eqref{Ptmfor}, there exist $C_2,C_3>0$ such that
\begin{align*}
\Im (u, (\tilde{P}_{k}-z)u)_{L^2}\leq&(u,\Op(H_{\tilde{p}}(k\log\jap{\x}))u)-\Im z\|u\|_{L^2}^2+C_2\|u\|_{H^{\frac{m-2}{2}+0}}^2\\
\leq&-C_0\|u\|_{H^{\frac{m-1}{2}}}^2-\Im z\|u\|_{L^2}^2+C_1\|Du\|_{L^2}^2+C_3\|u\|_{H^{\frac{m-2}{2}+0}}^2.
\end{align*}
Due to Cauchy-Schwarz inequality and the interpolation inequality, there exists $C_4>0$ such that
\begin{align*}
\|u\|_{H^{\frac{m-2}{2}+0}}^2\leq& \frac{C_0}{4C_3}\|u\|_{H^{\frac{m-1}{2}}}^2+C_4\|u\|_{H^{-N}}^2,\\
\left|\Im (u,(\tilde{P}_{k}-z)u)_{L^2} \right|\leq& \frac{C_0}{4}\|u\|_{H^{\frac{m-1}{2}}}^2+C_4\|(\tilde{P}_{k}-z)u\|_{H^{-\frac{m-1}{2}}}^2.
\end{align*}
By the support property of $d$ (Definition \ref{esces}) and the elliptic estimate of $\tilde{P}_k$, there exists $C_5>0$ such that
\begin{align*}
\|Du\|_{L^2}\leq C_5\|(\tilde{P}_{k}-z)u\|_{H^{-\frac{m-1}{2}}}+C_5\|u\|_{H^{-N}}\quad \text{for}\quad z\in I_L.
\end{align*}
Combining these four estimates, we obtain
\begin{align}\label{Freespf1}
\|u\|_{H^{\frac{m-1}{2}}}^2+(\Im z) C_6\|u\|_{L^2}^2\leq C_7\|(\tilde{P}_{k}-z)u\|_{H^{-\frac{m-1}{2}}}^2+C_8\|u\|_{H^{-N}}^2 \quad \text{for}\quad z\in I_L,
\end{align}
where $C_6=2C_0^{-1}>0$, $C_7=2(C_4+C_1C_5)C_0^{-1}>0$ and $2C_8=(C_1C_5+C_3C_4)C_0^{-1}>0$. Now we observe that $N>0$ implies $\|u\|_{H^{-N}}\leq \|u\|_{L^2}$. Therefore, we obtain $\|u\|_{H^{\frac{m-1}{2}}}^2\leq C_7\|(\tilde{P}_{k}-z)u\|_{H^{-\frac{m-1}{2}}}^2$ by taking $z$ as $\Im z>C_8/C_6$. This proves the last statement of this propsition.

Finally, we prove \eqref{Frees1} for $z\in I_L$.
The assumption $m>1$ and the Cauchy-Schwartz inequality imply
\begin{align*}
-\Im z\|u\|_{L^2}^2\leq (2C_6)^{-1}\|u\|_{H^{\frac{m-1}{2}}}^2+C_9\|u\|_{H^{-N}}^2
\end{align*}
with $C_9>0$ uniformly in $z\in I_L$. Then our estimate \eqref{Frees1} follows from this inequality and \eqref{Freespf1}.

\end{proof}

\begin{proof}[Proof of Theorem \ref{Fredcor}] 
First, we prove that $\Ker (\tilde{P}_k-z)<\infty$ is of finite dimension and $\Ran (\tilde{P}_k-z)$ is closed. Let $u_{\ell} \in\tilde{D}_{k}$ be a bounded sequence such that $(\tilde{P}_{k}-z)u_{\ell}$ is convergent in $H^{\frac{m-1}{2}}$. Due to \cite[Proposition 19.1.3]{Ho}, it suffices to prove that $u_{\ell}$ has a convergent subsequence in $\tilde{D}_{k}$. It easily follows from \eqref{Frees1} and the compactness of the inclusion $H^{\frac{m-1}{2}}\subset H^{-N}$ due to the Rellich-Kondrachov theorem. 

Next, we prove that the cokernel of $\tilde{P}_k-z$ is of finite dimension. To do this, it suffices to prove that the kernel of $(\tilde{P}_k-z)^*:  H^{\frac{m-1}{2}}\to \tilde{D}_{k}^*$ is of finite dimension.  
By the definition of the adjoint, we have
\begin{align*}
\Ker (\tilde{P}_k-z)^*=&\{u\in H^{\frac{m-1}{2}}\mid  (u,(\tilde{P}_k-z)w)_{L^2}=0,\, \forall w\in \tilde{D}_k \}\\
=&\{u\in H^{\frac{m-1}{2}}\mid  (u,(\tilde{P}_k-z)w)_{L^2}=0,\, \forall w\in C^{\infty}(M) \},
\end{align*}
where we use Lemma \ref{domaindense} in the second line. From this, $u\in H^{\frac{m-1}{2}}$ satisfies $(\tilde{P}_k-z)^*u=0$, then this equality holds in the distributional sense. The claim follows same as in the first half part of the proof. Therefore, $\tilde{P}_k-z$ is a Fredholm operator.

Due to Proposition \ref{Fredholmes}, we have $\|u\|_{H^{\frac{m-1}{2}}}\leq C\|(\tilde{P}_{k}-z)u\|_{H^{-\frac{m-1}{2}}}$ and $\|u\|_{H^{\frac{m-1}{2}}}\leq C\|(\tilde{P}_{k}-z)^*u\|_{H^{-\frac{m-1}{2}}}$ for $\Im z\gg 1$. This implies that $\Ker$ and $\mathrm{Coker}$ are the zero space. Hence $\tilde{P}_k-z$ is invertible for $\Im z\gg 1$. The analytic Fredholm theorem \cite[Theorem D.4]{Z} imply that the Fredholm index of $P_k-z$ is zero and there is a discrete set $S$ as desired.

\end{proof}

\section{The structure of the symbol}\label{sectionsymbol}

Now we return to the discussion in the case on the torus.
From now on, we  consider $P\in \Psi_{\mathrm{phg}}^m$ and $V\in \Psi_{\mathrm{phg}}^{m-\k}$ on $\tor$ satisfying the assumption in Theorem \ref{mainess}. Define
\begin{align*}
p=\s(P),\quad v=\s(V),
\end{align*}
which are real-valued by Proposition \ref{prireal}.

\subsection{The structure of the symbol}

Set $a_{\pm}(x):=p(x,\pm 1)$. Since $p$ is homogeneous of order $m$, we have 
\begin{align}\label{prepresent}
p(x,\x)=a_{\pm}(x)|\x|^m\quad \text{for}\quad x\in\tor,\,\, \pm \x\geq 1.
\end{align}
 for $x\in \tor$ and $\pm \x\geq 1$. 
We define
\begin{align}
&Z=Z_+\cup Z_-,\,\, Z_{\pm}=\{x\in \tor\mid a_{\pm}(x)=0\},\,\, Z_{\ast,\pm}=\{x\in \tor\mid a_{\ast}(x)=0,\, \pm  a_{\ast}'(x)>0\} \label{zero}
\end{align}
for $\ast\in \{\pm\}$. By the condition $(\ref{princond})$, it turns out that $Z$ is a finite set.

\begin{lem}\label{prilem}

\noindent$(i)$ $p$ is elliptic if and only if $Z=\emptyset$.

\noindent$(ii)$ Then $p$ is real principal type if and only if 
\begin{align}\label{princond}
a_{\pm}(x)= 0\Rightarrow a_{\pm}'(x)\neq 0.
\end{align}

\end{lem}

\begin{proof}

\noindent $(i)$ This part is easily proved.

\noindent $(ii)$
This part follows from the following observation:
\begin{align*}
&(x,\x)\in \{p=0\}\Leftrightarrow a_{\pm}(x)=0,\\
&(x,\x)\in \{dp\neq 0\}\Leftrightarrow a_{\pm}(x)\neq 0\,\, \text{or}\,\, a_{\pm}'(x)\neq 0.
\end{align*}

\end{proof}

\begin{rem}
Actually, we have
\begin{align}\label{zeroeq}
\# Z_{\ast, +}=\# Z_{\ast, -}\quad \text{for}\quad \ast\in \{\pm\}.
\end{align}
To see this, let $x\in Z_{\ast, +}$. We consider this problem on the universal cover (and identify the function $a_{*}$ as a periodic function on $\re$). Let 
\begin{align*}
y=\min \{z\in (x,\infty)\mid a_{\ast}(z)=0\}.
\end{align*}
Since $a_{\ast}$ is periodic, we have $y\leq x+2\pi$. It follows from $x\in Z_{\ast,+}$ that $a_{\ast}(z)>0$ for $z\in (x,y)$. Since $a_*'(y)\neq 0$, then $a_{\ast}$ does not attain the minimum at $y$. This implies $a_*(z)<0$ on a small interval $(y,y+\e)$. Hence we have $a_*'(y)<0$. Consequently, we obtain $\#Z_{\ast, +}\leq \# Z_{\ast, -}$. The inequality $\#Z_{\ast, -}\leq \# Z_{\ast, +}$ is similarly proved.
 We also note that the real principal type condition is necessary for $(\ref{zeroeq})$. In fact, the function $a(x)=1+\sin x$ has a unique zero at $x=\frac{3}{2}\pi$, however, the function $a$ satisfies $a(\frac{3}{2}\pi)=a'(\frac{3}{2}\pi)=0$.

\end{rem}

\subsection{Classical trajectories for real principal type operators}\label{subseciandiii}

In this subsection, we shall show that the equivalence $(i)\Leftrightarrow (iii)$ in Theorem \ref{mainess}.

\begin{proof}[Proof of Theorem \ref{mainess}, $(i)\Leftrightarrow (iii)$]

First, we show $(iii)\Rightarrow (i)$ by a contradiction argument. We assume that $(iii)$ holds and there exists an integral curve $(y(t),\y(t))$ of the vector field $H_p$ with maximal extension time $T<\infty$. Since $\tor$ is closed, this implies $|\y(t)|\to \infty$ as $t\to T$ due to the maximal extension theorem of ODE. 

If $m\leq 1$ holds, then $C:=\sup_{(x,\x)\in T^*\tor}(1+|\x|)|\pa_{x}p(x,\x)|<\infty$. On the other hand,  by the second equation of the Hamilton equation $\y'(t)=-(\pa_xp)(y(t),\y(t))$,
we have
\begin{align*}
|\y'(t)|\leq C(1+|\y'(t)|).
\end{align*}
This contradicts to the assumption $|\y(t)|\to \infty$ as $t\to T$ and the Gronwall inequality.

If $p$ is elliptic, then $|p(x,\x)|\geq C_1|\x|^m-C_2$ by definition. The conservation law $p(y(t),\y(t))=p(y(0),\y(0))$ implies $C_1|\y(t)|^m\leq p(y(0),\y(0))+ C_2$ for all $t$, which contradicts to the assumption $|\y(t)|\to \infty$ as $t\to T$. This completes the proof of $(iii)\Rightarrow (i)$.

Next, we prove $(i)\Rightarrow (iii)$. By Lemma \ref{prilem}, it suffices to find an incomplete trajectory of $H_0$ assuming that $m>1$ and $Z\neq \emptyset$. We deal with the case of $Z_{+,+}\neq \emptyset$ only and the other cases are similarly treated.

Let $x_0\in Z_{+,+}$ and consider the trajectory $(y(t), \y(t))$ with initial data $(y(0),\y(0))=(x_0,2)$. Since $a(x_0)=0$ and the first equation of the Hamilton equation is given by $y'(t)=(\pa_{\x}p)(y(t),\y(t))$, we have $y(t)=x_0$ for all $t$ by Lemma \ref{prilem}. Moreover, $-\pa_{x}p(x_0,\x)=-a_+'(x_0)|\x|^m\leq 0$ for $\x\geq 1$ by the definition of $Z_{+,+}$.
By the second equation of the Hamilton equation $\y'(t)=-(\pa_xp)(x_0,\y(t))$, it follows that $|\y(t)|\geq 1$ for $t\leq 0$.
In particular, $\y'(t)=-a_+'(x_0)|\y(t)|^m$ for $t\leq 0$ by Lemma \ref{prilem}. Since $a'(x_0)>0$ and $m>1$, this equation blows up at a negative finite time. This completes the proof of $(i)\Rightarrow (iii)$.

\end{proof}

\begin{rem}
The proof of $(iii)\Rightarrow (i)$ works on general closed manifolds.

\end{rem}

\subsection{Construction of an escaping function}\label{escapesubst}

We define conic Lagrangian submanifolds $\L_{\mathrm{inc}}/\L_{\mathrm{out}}$ of $T^*\tor$ by
\begin{align*}
\L_{\mathrm{inc}}:=\left(Z_{+,+}\times (0,\infty)\right)\cup \left(Z_{-,-}\times (-\infty,0)\right),\,\, \L_{\mathrm{out}}:=\left(Z_{+,-}\times (0,\infty)\right)\cup \left(Z_{-,+}\times (-\infty,0)\right),
\end{align*}
where we identify $T^*\tor$ as $\tor \times \re$ and $Z_{\ast,\pm}$ are defined in \eqref{zero}. We note that $\L_{\mathrm{inc}}\cup \L_{\mathrm{out}}$ corresponds to the characteristic set of $p$. In other word, $p$ is elliptic away from there.

\begin{rem}
It turns out that $\L_{\mathrm{inc}}$ (resp. $\L_{\mathrm{out}}$) or more precisely its fiber infinity is a radial source (resp. a radial sink) of $P$ in the sense of \cite[Definition E.50]{DZ} although we do not use this fact here. One of the important property of $\L_{\mathrm{inc}}/\L_{\mathrm{out}}$ is that $H_p$ is parallel to the radial vector field $\x\pa_{\x}$:
\begin{align*}
H_p|_{(\L_{\mathrm{inc}}\cup \L_{\mathrm{out}})\cap \{|\x|\geq 1\}}=(-a'(x) \sgn\x)\cdot |\x|^{m-1}\cdot\left(\x\pa_{\x}\right)
\end{align*}
and its coefficient $-a'(x) \sgn\x$ is positive on $\L_{\mathrm{inc}}$ and negative on $\L_{\mathrm{out}}$.

\end{rem}

In our case, it is easy to construct an escaping function of $P+V$.

\begin{lem}\label{escexist}
Let $\e_0>0$ and $k\in S^0$ such that $|k(x,\x)|\leq \e_0$ holds everywhere and 
\begin{align*}
k(x,\x)=\begin{cases}
\e_0\quad \text{near} \quad (x,\x)\in \L_{\mathrm{inc}}\cap \{|\x|\geq 1\},\\
-\e_0\quad \text{near} \quad (x,\x)\in \L_{\mathrm{out}}\cap \{|\x|\geq 1\}.
\end{cases}
\end{align*}
Then $k$ is an escaping function of $P+V$ in the sense of Definition \ref{esces}.

\end{lem}

\begin{proof}
It suffices to prove that there exist $C>0$ and $R>0$ such that
\begin{align}\label{escexistpf1}
H_{p+v}(k(x,\x)\log\jap{\x})\leq -C\jap{\x}^{m-1}
\end{align}
if $(x,\x)$ is away from $(\L_{\mathrm{inc}}\cup \L_{\mathrm{out}}) \cap \{\pm \x\geq 1\}$ and $|\x|\geq R$ since $p$ is elliptic away from $\L_{\mathrm{inc}}\cup \L_{\mathrm{out}}$.

We write $p(x,\x)=a_{\pm}(x)|\x|^m$ for $\pm\x\geq 1$ as in Lemma \ref{prilem}. 
By Leibniz's rule, we have $H_p(k(x,\x)\log\jap{\x})=k(x,\x)H_p\log\jap{\x}+(H_pk)(x,\x)\log\jap{\x}$.
 By the choices of $k$ and the definition of $\L_{\mathrm{inc}},\L_{\mathrm{out}}$, there exists $c>0$ such that
\begin{align*}
k(x,\x)H_p\log\jap{\x}=-k(x,\x)a_{\pm}'(x)|\x|^m\x\jap{\x}^{-1}\leq -c\jap{\x}^{m-1}
\end{align*}
near $(\L_{\mathrm{inc}}\cup \L_{\mathrm{out}}) \cap \{\pm \x\geq 1\}$. On the other hand, $H_pk$ is supported away from $(\L_{\mathrm{inc}}\cup \L_{\mathrm{out}}) \cap \{\pm \x\geq 1\}$. Moreover, we have $H_v(k\log\jap{\x})\in S^{m-1-\k+0}$, which implies the existence of $C>0$ such that $H_v(k(x,\x)\log\jap{\x})\leq C\jap{\x}^{m-1-\frac{\k}{2}}$. Taking $R>0$ sufficiently large, we have $H_v(k(x,\x)\log\jap{\x})\leq \frac{c}{2}\jap{\x}^{m-1}$ for $|\x|\geq R$. Thus \eqref{escexistpf1} holds for $C=\frac{c}{2}$.

\end{proof}

\section{Relationship between essential self-adjointness and the order of the operator}\label{firstpf}
In this section, we apply the method developed in \cite{T} to our operator $P$ and prove Theorem \ref{mainess} $(ii)\Leftrightarrow (iii)$.

\subsection{Proof of Theorem \ref{mainess} $(iii)\Rightarrow (ii)$}\label{subseciiitoii}

If $p$ is elliptic, then the elliptic regularity of $P+V$ implies $D_{\mathrm{max}}(P+V)\subset H^{m}(\tor)$, where $D_{\mathrm{max}}$ is defined in \eqref{Dmaxdef}. Now the essential self-adjointness for $P+V$ on $C^{\infty}(\tor)$ follows from Proposition \ref{esscri1}. 

If $m\leq 1$, then Proposition \ref{esscri1} implies the essential self-adjointness for $(P+V)|_{C^{\infty}(\tor)}$ (see also \cite[Lemma 4.1]{CB} or \cite[Lemma A.1]{FS}).  This completes the proof of Theorem \ref{mainess} $(iii)\rightarrow (ii)$.

In the following of this section, we shall prove Theorem \ref{mainess} $(ii)\rightarrow (iii)$. To do this, assuming that $m>1$ and that $p$ is not elliptic, we construct non-trivial $L^2$-eigenfunctions for $P$ with non-real eigenvalues in the remainder of this section.

\subsection{WKB construction}

In the rest of this section, we assume
\begin{align*}
m>1.
\end{align*}

In this subsection, we construct an approximate eigenfunction of $P$ whose wavefront set lies in the incoming region (the radial source) of $p$. The following construction below is more or less standard in WKB theory when $V=0$. The treatment of the long-range perturbation $V$ is inspired by that of the spherical waves in the long-range scattering theory \cite{GY} (see also \cite{IK} for the construction of the plane waves, so called the Isozaki-Kitada parametrix).
In \cite{W}, a similar construction on closed surfaces is given when $\k=1$.

Define
\begin{align}\label{deltadef}
\d:=\min(2\k,m)>1
\end{align}
where we recall that $m>1$ $\k>\frac{1}{2}$ are real numbers such that $P\in \Psi^m_{\mathrm{phg}}$ and $V\in \Psi^{m-\k}_{\mathrm{phg}}$.

\begin{prop}\label{WKBthm}
Let $z\in \mathbb{C}$ and $x_0\in Z_{++}$, where $Z_{++}$ is defined in \eqref{zero}. Then there exists $u_{\mathrm{inc},z}\in C^{\infty}(\tor\setminus \{x_0\})\cap H^{\frac{m-1}{2}-0}$ such that
\begin{align}\label{u_-reg}
(P+V-z)u_{\mathrm{inc},z}\in H^{-\frac{m+1}{2}+2\d-0},\,\, u_{\mathrm{inc},z}\notin H^{\frac{m-1}{2}},\,\, \mathrm{WF}^{\frac{m-1}{2}}(u_{\mathrm{inc},z})= \L_{x_0,+},
\end{align}
where $\L_{x_0,+}\subset \L_{\mathrm{inc}}$ is defined in \eqref{simpleLag}.

\end{prop}

\begin{rem}

$(1)$
Similar results hold for $x_0\in Z_{-,-}$ replacing the wavefront condition by
\begin{align*}
\mathrm{WF}(u)\subset \L_{x_0,-}.
\end{align*}

$(2)$ Under our assumption $m>1$, the spectral parameter $z$ can be regarded as a short-range (subsubprincipal) perturbation.

\end{rem}

\begin{proof}

We find $u_{\mathrm{inc},z}\in  I^{-\frac{m}{2}+\frac{1}{4}}_{\k}(\tor,\L_{x_0,+})\setminus H^{\frac{m-1}{2}}$ such that 
\begin{align*}
(P+V-z)u_{\mathrm{inc},z}\in I^{\frac{m}{2}+\frac{1}{4}-\d}_{\k}(\tor,\L_{x_0,+}).
\end{align*} 
By Corollary \ref{stdsubprireal}, the subprincipal symbol $\s_{\mathrm{sub}}(P)$ of $P$ is real-valued in a standard coordinate, where we recall that a standard coordinate is defined in Definition \ref{stdcoord} \footnote{If we consider the operator acting on half-densities instead of functions, then the subprincipal symbol is always real-valued by \cite[Theorem 18.1.34]{Ho}. Here, we avoid using half-densities.}. We define a real-valued function $q$ by
\begin{align*}
q:=\s_{\mathrm{sub}}(P)+v\in S^{m-\k}
\end{align*}
in standard coordinates, where we recall $v=\s(V)\in S^{m-\k}$.

We take $b\in S_{\k,+}^{-\frac{m}{2}}(\re)$ as 
\begin{align}\label{b_0deftr}
b(\x):=\x^{-\frac{m}{2}}e^{i\g(\x)},\quad \g(\x):=\frac{1}{a_+'(x_0)}\int_1^{\x}\frac{q(x_0,\y)}{\y^m}d\y \quad \text{for}\,\,  \x\geq 1,
\end{align}
where we recall that $a_+$ is defined in \eqref{prilem}. We note that the symbolic property of $b_0$ is a consequence of $\g\in S^{1-\k}$ and the fact that $\g$ is real-valued.
Moreover, it is easy to see that $b_0$ satisfies the first transport equation
\begin{align}\label{b_0transeq}
i^{-1}(\mathcal{L}_{H_p}b)(\x)+q(x_0,\x)b(\x)=0\quad \text{for}\,\, \pm \x\geq 1,
\end{align}
where we recall $(\mathcal{L}_{H_p}b)(\x)=-(\pa_xp)(x_0,\x)\pa_{\x}b(\x)-\frac{1}{2}(\pa_x\pa_{\x}p)(x_0,\x)b(\x)$ from Lemma \ref{Laglem}. 
Now we define $u_{\mathrm{inc},z}$ by
\begin{align}\label{u_0rep}
u_{\mathrm{inc},z}(x)=\int_{\re}b(\x)e^{i(x-x_0)\x}d\x
\end{align} 
in a standard coordinate near $x_0$ and $u_{\mathrm{inc},z}$ is a smooth function away from $x_0$. Then we have $u_{\mathrm{inc},z}\in  I^{-\frac{m}{2}+\frac{1}{4}}_{\k}(\tor,\L_{x_0,+})$ by the definition of Lagrangian distributions. 

Next, we see that $(P+V-z)u_{\mathrm{inc},z}\in  I^{\frac{m}{2}+\frac{1}{4}-\d}_{\k}(\tor,\L_{x_0,+})$. Since
\begin{align*}
zu_{\mathrm{inc},z}\in  I^{-\frac{m}{2}+\frac{1}{4}}_{\k}(\tor,\L_{x_0,+})=I^{\frac{m}{2}+\frac{1}{4}-m}_{\k}(\tor,\L_{x_0,+})\subset  I^{\frac{m}{2}+\frac{1}{4}-\d}_{\k}(\tor,\L_{x_0,+})
\end{align*}
by the choice of $\d>1$: \eqref{deltadef}, it suffices to prove $(P+V)u_{\mathrm{inc},z}\in  I^{-\frac{m}{2}+\frac{1}{4}-\d}_{\k}(\tor,\L_{x_0,+})$. By Lemma \ref{Laglem}, there exists $\tilde{w}_1\in I^{\frac{m}{2}+\frac{1}{4}-2\k}_{\k}(\tor,\L_{x_0,+})$ such that 
\begin{align*}
(P+V)u_{\mathrm{inc},z}=\int_{\re}(i^{-1}\left(\mathcal{L}_{H_p}b)(\x)+q(x_0,\x)b(\x)\right)e^{i(x-x_0)\x}d\x+\tilde
{w}_1
\end{align*}
in the same standard coordinate. Since the integrant of the first term is zero for $|\x|\geq 1$ by the choice of $b$, we conclude $(P+V)u_{\mathrm{inc},z}\in  I^{\frac{m}{2}+\frac{1}{4}-2\k}_{\k}(\tor,\L_{x_0,+})\subset  I^{\frac{m}{2}+\frac{1}{4}-\d}_{\k}(\tor,\L_{x_0,+})$ by \eqref{deltadef}.

By the property of the Lagrangian distributions (Lemma \ref{Laglem}), we see that $u_{\mathrm{inc},z}\in C^{\infty}(\tor\setminus \{x_0\})\cap H^{\frac{m-1}{2}-0}$, $(P+V)u_{\mathrm{inc},z}\in H^{-\frac{m+1}{2}+2\d-0}$ and $\mathrm{WF}(u_{\mathrm{inc},z})\subset \L_{x_0,+}$. Thus we only need to prove 
that $u_{\mathrm{inc},z}\notin H^{\frac{m-1}{2}}$. Since $u_{\mathrm{inc},z}$ is smooth except at $x_0$, it suffices to that $u_{\mathrm{inc},z}\notin H^{\frac{m-1}{2}}$ near $x_0$. By the fundamental property of the Fourier transform and \eqref{u_0rep}, this is equivalent to $b\notin \jap{\x}^{-\frac{m-1}{2}}L^2(\re)$, which directly follows from \eqref{b_0deftr}. This completes the proof.

\end{proof}

\subsection{Existence of generalized eigenfunctions}

In this subsection, we construct generalized eigenfunctions of $P+V$ in order to show Theorem \ref{mainess} $(ii)\Rightarrow (iii)$.

\begin{prop}\label{eigenconst}
Let $S$ be the discrete set constructed in Theorem \ref{Fredcor} and $z\in \mathbb{C}\setminus S$. Then there exists $u_z\in L^2\setminus\{0\}$ such that $(P+V-z)u_z=0$ in the distributional sense.
\end{prop}

\begin{rem}
More strongly, we have $u_z\in H^{\frac{m-1}{2}-0}$ by Proposition \ref{propradsink}.
\end{rem}

\begin{proof}
Let $\e_0>0$ small enough such that
\begin{align}\label{epsilon}
\e_0<\min\left(\frac{m-1}{2},2\d-1 \right),
\end{align}
where $\d>\frac{1}{2}$ is defined in \eqref{deltadef}. Then we can take an escaping function $k$ of $P+V$ by Lemma \ref{escexist} for such $\e$. We define
\begin{align*}
R_+(z):=G_k^{-1} (P_k-z)^{-1}G_k,\quad z\in\mathbb{C}\setminus S
\end{align*}
where $G_k$ is defined just before \eqref{Ptmdef} and $(P_k-z)^{-1}: H^{-\frac{m-1}{2}} \to \tilde{D}_k$ is defined in Theorem \ref{Fredcor}.

Let $u_{\mathrm{inc},z}\neq 0$ as in Proposition \ref{WKBthm}. We observe that $(P+V-z)u_{\mathrm{inc},z}\in H^{-\frac{m}{2}-1+2\d-0}$ by Theorem \ref{WKBthm} and $k(x,\x)\in [-\e_0,\e_0]$. Then Lemma \ref{Anistrolem} $(i)$ yields
\begin{align*}
G_k(P+V-z)u_{\mathrm{inc},z}\in H^{-\frac{m+1}{2}+2\d-\e-0}\subset H^{-\frac{m-1}{2}},
\end{align*}
due to the choices of $\e$ and $\d$: \eqref{deltadef} and \eqref{epsilon}.
 Hence 
\begin{align*}
u_{\mathrm{out},z}:=-R_+(z)(P+V-z)u_{\mathrm{inc},z},\quad u_z:=u_{\mathrm{out},z}+u_{\mathrm{inc},z}.
\end{align*}
are well-defined and satisfy
\begin{align*}
(P+V-z)u_{z}=0.
\end{align*}
We observe that the range of $(P_k-z)^{-1}$ is $\tilde{D}_k$ and $\tilde{D}_k\subset H^{\frac{m-1}{2}}$, which imply
\begin{align}\label{G_kuoutreg}
G_ku_{\mathrm{out},z}\in H^{\frac{m-1}{2}}.
\end{align}
Combining $|k(x,\x)|\leq \e_0< \frac{m-1}{2}$ and Lemma \ref{Anistrolem} $(i)$, we obtain $u_{\mathrm{out},z}\in L^2$. Since $u_{\mathrm{inc},z}\in H^{\frac{m-1}{2}-0}\subset  L^2$ by Proposition \ref{WKBthm} and $m>1$, we conclude
\begin{align*}
u_z\in L^2.
\end{align*}

It remains to prove that $u_z\neq 0$. We recall from Theorem \ref{WKBthm} that $\mathrm{WF}(u)\subset \L_{x_0,+}$ with $x_0\in Z_{+,+}$. We observe that $k(x,\x)=\e_0$ near $\L_{x_0,+}\cap \{|\x|\geq 1\}$ by Lemma \ref{escexist}. 
Due to \eqref{G_kuoutreg}, we can employ Lemma \ref{Anistrolem} $(iv)$ with $\L=\L_{x_0,+}$. Then we obtain $u_{\mathrm{out},z}\in H^{\frac{m-1}{2}+\e}\subset H^{\frac{m-1}{2}}$ microlocally near $\L_{x_0,+}$, that is $\mathrm{WF}^{\frac{m-1}{2}}(u_{\mathrm{out},z})\cap \L_{x_0,+}=\emptyset$. On the other hand, $\mathrm{WF}^{\frac{m-1}{2}}(u_{\mathrm{inc},z})=\L_{x_0,+}$ by \eqref{u_-reg}. Therefore, $u_z=u_{\mathrm{out},z}+u_{\mathrm{inc},z}\neq 0$. This completes the proof.

\end{proof}

\subsection{Proof of Theorem \ref{mainess} $(ii)\Rightarrow (iii)$}\label{subseciitoiii}

It suffices to show that $m>1$ implies $(P+V)|_{C^{\infty}(\tor)}$ is not essentially self-adjoint. By Proposition \ref{eigenconst}, for each $z\in \mathbb{C}\setminus S$, there exists $u_z\in L^2\setminus \{0\}$ such that $(P+V-z)u_z=0$ in the distributonal sense. Since $S$ is a discrete set, then Proposition \ref{Esscricor} shows that $(P+V)|_{C^{\infty}(\tor)}$ is not essentially self-adjoint. This proves $(ii)\Rightarrow (iii)$ in Theorem \ref{mainess}.

\section{Proof of Corollary \ref{examplecor}}\label{Sectionproofcor}

\begin{proof}[Proof of Corollary \ref{examplecor} $(i)$]
By using Fourier expansion in $y$, we have
\begin{align*}
\Box_g=\oplus_{n\in \mathbb{Z}}(\a\pa_{x}(\sin x\pa_x)-2nD_{x})\quad \text{on}\quad \oplus_{n\in \mathbb{Z}}L^2(\tor_x).
\end{align*}
From Theorem \ref{mainess}, it follows that $\a\pa_{x}(\sin x\pa_x)$ (with $\a\in \re\setminus \{0\}$) is not essentially self-adjoint on $C^{\infty}(\tor)$ and hence have a non-trivial $L^2$-eigenfunction $u$ with the eigenvalue $i$. Then it turns out that the function $w(x,y)=u(x)\in L^2(\tor^2)$ satisfies $(\Box_g-i)w=0$. By Proposition \ref{esscrieigen}, we conclude that $\Box_g$ is not essentially self-adjoint on $C^{\infty}(\tor^2)$.
\end{proof}

\begin{proof}[Proof of Corollary \ref{examplecor} $(ii)$]

We consider the Lorentzian metric $g=(x^2+y^2)^{-1}(2dxdy)$ on $\re^2\setminus \{0\}$. For $\l>0$,
\begin{align*}
M:=(\re^2\setminus \{0\})/\sim,\quad x\sim y\Leftrightarrow \exists  k\in\mathbb{Z}\,\, \text{such that}\,\, x=\l^ky.
\end{align*}
Then $g$ induces a Lorentzian metric on $M$ since $g_{\l(x,y)}=(\l^2x^2+\l^2y^2)^{-1}(2d(\l x)d(\l y))=g_{(x,y)}$ for $\l>0$. Moreover, $M$ is homeomorphic to a two-dimensional torus. We call the manifold $(M,g)$ the Clifton-Pohl torus. We assume $\l=e$ for notational simplicity and note that $\l=2$ is assumed in \cite[p.193]{O}.

We use the coordinate
\begin{align*}
x=e^t\cos\theta,\quad y=e^t\sin\theta,
\end{align*}
where $t\in [0,1]$ and $\theta\in[0,2\pi)$. We observe that the $C^{\infty}$ structure in this coordinate is more natural than the $(x,y)$-coordinate since $f\in C^{\infty}(M)$ if and only if $f(t+k,\theta+ 2\pi \ell)=f(t,\theta)$ for all $k,\ell\in\mathbb{Z}$.
Since $dx=(e^t\cos\theta) dt-(e^t\sin\theta)d\theta$ and $dy=(e^t\sin\theta) dt+(e^t\cos\theta)d\theta$, we have
\begin{align*}
g=&(2\cos\theta\sin\theta)dt^2+2(\cos^2\theta-\sin^2\theta)dtd\theta-2(\sin\theta\cos\theta) d\theta\\
=&(\sin2\theta)dt^2+2(\cos 2\theta)dtd\theta-(\sin 2\theta)d\theta^2
\end{align*}
and $\det g(t,\theta)=-1$.
Thus, the wave operator is given by
\begin{align*}
\Box_g=(\sin 2\theta)\pa_t^2+\pa_t\circ \cos2\theta \circ \pa_{\theta}+\pa_{\theta}\circ \cos2\theta \circ \pa_{t}-\pa_{\theta}(\sin 2\theta \pa_{\theta}).
\end{align*}
The $L^2$-space is $L^2([0,1]\times [0,2\pi];dtd\theta)$.

Since the operator $-\pa_{\theta}(\sin 2\theta \pa_{\theta})$ on $\mathbb{T}_{\theta}$ is not essentially self-adjoint on $C^{\infty}(\mathbb{T})$ by Theorem \ref{mainess}, we can take $u_{\pm}\in L^2(\mathbb{T})\setminus \{0\}$ such that
\begin{align*}
(-\pa_{\theta}(\sin 2\theta \pa_{\theta})\pm i)u_{\pm}(\theta)=0.
\end{align*}
due to Proposition \ref{esscrieigen}. Now we define $w_{\pm}(t,\theta):=u_{\pm}(\theta)$, which belongs to $L^2([0,1]\times [0,2\pi];dtd\theta)\setminus \{0\}$ and $(\Box_g\pm i)w_{\pm}=0$ since $w_{\pm}$ is independent of the $t$-variable and
\begin{align*}
\Box_gw_{\pm}=-\pa_{\theta}(\sin 2\theta \pa_{\theta})w_{\pm}.
\end{align*}
Using Proposition \ref{esscrieigen} again, we conclude that $\Box_g$ is not essentially self-adjoint on $C^{\infty}(M)$.

\end{proof}

\section{Spectral properties for real principal type operators}\label{secondpf}

\subsection{Absolute continuity of the spectrum}\label{subsecMourre}

In this subsection, we prove Theorem \ref{spthm} $(i)$ by using the Mourre theory \cite{M,Ge}.

\subsection*{Mourre theory}

Let us briefly review the Mourre theory.

If $H,A$ are self-adjoint operators, then $[H,A]$ is defined as a quadratic form on $D(H)\cap D(A)$ by $(u,[H,A]w):=(Hu,Aw)-(Au,Hw)$ for $u,w\in D(H)\cap D(A)$. If we further assume that $A$ is a bounded operator, then $(H-i)^{-1}[H,A](H-i)^{-1}$ can be regarded as a bounded operator on $\mathcal{H}$ and $[A,(H-i)^{-1}]=(H-i)^{-1}[H,A](H-i)^{-1}$ holds.

\begin{defn}
Let $A$ be a self-adjoint operator and $B$ be a bounded operator on a Hilbert space $\mathcal{H}$. We call $B\in C^1(A)$ if a quadratic form $[A, B]$ on $D(A)\times D(A)$ can be extended to a bounded operator from $D(A)$ to $\mathcal{H}$. We call $B\in C^2(A)$ if $B\in C^1(A)$ and $[A,B]\in C^1(A)$.

Let $H$ be a (not necessary bounded) self-adjoint operator. For $k=1,2$, we call $H\in C^k(A)$ if $(H-i)^{-1}\in C^k(A)$.

\end{defn}

 Let $H,A$ be self-adjoint operators on a Hilbert space $\mathcal{H}$.

\begin{thm}\cite{ABG,Ge,M}\label{Mourrethm}
Suppose $H\in C^2(A)$. We further assume that for a compact interval $I\Subset \re$, $E_{H}(I)[H,iA] E_{H}(I)\geq  c E_{H}(I)+K$ with $c>0$ and a compact operator $K\in B(\mathcal{H})$, where $E_H(I)$ is the spectral projection of $H$ of the interval $I$. Then the spectrum of $H$ on $I$ is absolutely continuous except a finite number of eigenvalues.
\end{thm}

Actually, weighted resolvent estimates are proved there.
For the relationship between the absolute continuity of the spectrum and the resolvent estimate, see \cite[Vol IV,Theorem XIII.20]{RS}.

\begin{lem}\cite[Theorem 6.3.4]{ABG}\label{C^1cri2}
Assume that $e^{itA}$ leaves $D(H)$ for each $t\in \re$. If the quadratic form $[H,A]$ on $D(H)\cap D(A)$ can be extended to a bounded operator $[H,A]_{\mathrm{ext}}$ from $D(H)$ to $D(H)^*$, then $H\in C^1(A)$. If in addition $[H,A]_{\mathrm{ext}}$ is bounded from $D(H)$ to $\mathcal{H}$ and the quadratic form $[[H,A]_{\mathrm{ext}},A]$ on $D(H)\cap D(A)$ can be extend to a bounded operator $[[H,A]_{\mathrm{ext}},A]_{\mathrm{ext}}$ from $D(H)$ to $D(H)^*$, then $H\in C^2(A)$.

\end{lem}



\subsection*{Mapping properties of Schr\"odinger propagators}

To employ the criterion Lemma \ref{C^1cri2} in our situation, some information of the Schr\"odinger propagator $e^{itA}$ is needed.

\begin{lem}\label{Schpreserve}
Let $M$ be a closed manifold and $\mathcal{H}=L^2(M;d\m)$ with a smooth density $d\m$ on $M$. Suppose that $A\in \Psi^1$ and $A=A^*$ on $C^{\infty}(M)$. We denote the unique self-adjoint extension of $A$ by the same symbol $A$ (see Proposition \ref{esscri1}). Then the propagator $e^{itA}$ preserves $C^{\infty}(M)$ for each $t\in \re$.

\end{lem}

\begin{proof}
It suffices to show that $e^{itA}$ preserves $H^k$ for all $k>0$. Fix a Riemannian metric $g$ on $M$ and
\begin{align*}
N=(I-\Delta_g)^{\frac{k}{2}},\quad N_{\e}=(I-\Delta_g)^{\frac{k}{2}}(I-\e\Delta_g)^{-\frac{k+1}{2}}.
\end{align*}
Then $N\in \Psi^k$ is an elliptic self-adjoint operator with domain $H^{k}$ and $\{N_{\e}\}_{0<\e\leq 1}\subset \Psi^{-1}$ is a family of bounded self-adjoint operators such that $\{N_{\e}\}_{0<\e\leq 1}\subset \Psi^{-1}$ is uniformly bounded in $\Psi^k$ and $N_{\e}\to N$ in $\Psi^{k+0}$.

Let $u\in H^k$ and show $e^{itA}u\in H^k$. We observe that $v_{\e}(t):=N_{\e}e^{itA}u$ belongs to $D(A)$ since $H^1\subset D(A)$ (by $A\in \Psi^1$) and $N_{\e}\in \Psi^{-1}$. Then $v_{\e}(t)$ satisfies 
\begin{align*}
i\pa_tv_{\e}(t)+(A+B_{\e})v_{\e}(t)=0,\quad v_{\e}(t)|_{t=0}=N_{\e}u,
\end{align*}
where $B_{\e}:=[N_{\e},A]N_{\e}^{-1}$ is uniformly bounded in $\Psi^0$. Since $v_{\e}(t)\in D(A)$ and $A$ is self-adjoint, we have
\begin{align*}
\frac{d}{dt}\|v_{\e}(t)\|_{\mathcal{H}}^2=2\Re(v_{\e}(t),iB_{\e}v_{\e}(t))_{\mathcal{H}}\leq C\|v_{\e}\|_{\mathcal{H}}^2
\end{align*}
with $C>0$ independent of $0<\e\leq 1$. Then Gronwall's inequality implies that $\{v_{\e}(t)\}_{0<\e\leq 1}$ is uniformly bounded in $L^2$. On the other hand, since $N_{\e}\to N$ in $\Psi^{k+0}$, we have $v_{\e}(t)\to Ne^{itA}u$ as $\e\to 0$ in the distributional sense. These imply $Ne^{itA}u\in \mathcal{H}$. By the ellipticity of $N$ (and $e^{itA}u\in \mathcal{H}$), we conclude $e^{itA}u\in H^k$. This completes the proof.

\end{proof}

\subsection*{Application of the Mourre theory}

Now let $P,V$ be satisfying the assumption of Theorem \ref{mainess}. Suppose that $0<m\leq 1$ holds and $p$ is not elliptic.
We write $H=P+V$.
Take the symmetric operator $A$ as
\begin{align*}
A:=\Op(a)\in\begin{cases}
\Psi^{1-m}\,\, \text{if}\,\, 0< m<1,\\
\Psi^{+0}\,\, \text{if}\,\, m=1,
\end{cases}\,\,
a(x,\x)=\begin{cases}
-k(x,\x)\jap{\x}^{1-m}\,\, \text{if}\,\, 0< m<1,\\
-k(x,\x)\log\jap{\x}\,\, \text{if}\,\, m=1,
\end{cases}
\end{align*}
where $k\in S^0$ is as in Lemma \ref{escexist}.
By Theorem \ref{mainess}, it turns out that $H$ and $A$ are essentially self-adjoint on $C^{\infty}(\tor)$. We denote the unique self-adjoint extensions of $H$ and $A$ by the same symbols $H$ and $A$ respectively. In order to control the lower order term in the Mourre inequality and to prove the $A$-regularity of $H$, we need the following standard lemma.

\begin{lem}\label{Mourrelow}
Let $k<m$ and $b\in S^{k}$ be supported away from $\left(\L_{\mathrm{inc}}\cup \L_{\mathrm{out}}\right)\cap  \{|\x|\geq 1\}$, where $\L_{\mathrm{inc}},\L_{\mathrm{out}}$ are defined in Subsection \ref{escapesubst}. Then
Then the operator $\Op(b)(H-i)^{-1}$ is bounded and compact on $L^2$. In particular, $\Op(b)E_I(H)$ is a compact operator for $I\Subset \re$.
\end{lem}

\begin{proof}
We set $c(x,\x)=b(x,\x)(p(x,\x)+v(x,\x)-i)^{-1}$, where we recall that $p=\s(P)$ and $v=\s(V)$. By the support property of $b$, $p+v$ is elliptic on $\supp b$, which implies $c\in S^{k-m}$. Then we have
\begin{align*}
\Op(c)(H-i)=\Op(b)+R,\quad R\in \Psi^{k-1}.
\end{align*}
Since $k<m\leq 1$, $\Op(c)$ and $R$ are bounded and compact operators on $L^2$. Hence $\Op(b)(H-i)^{-1}=\Op(c)-R(H-i)^{-1}$ is a bounded and compact operator on $L^2$.

\end{proof}

Now we prove the $A$-regularity of $P+V$.
\begin{lem}\label{P+VC^2}
$P+V\in C^2(A)$.
\end{lem}

\begin{proof}

First, we show that $[H,A](H-i)^{-1},[[H,A],A]\in B(L^2)$. Since $m\leq 1$, we have $[[H,A],A]\in \Psi^0$ and hence $[[H,A],A]\in B(L^2)$. For $m<1$, $[H,A]\in \Psi^0$ and $[H,A]\in B(L^2)$. To prove $[H,A](H-i)^{-1}\in B(L^2)$ for $m=1$, we observe
\begin{align}\label{conjder}
H_{p}a(x,\x)&=-(H_pk)(x,\x)\log\jap{\x}-k(x,\x)H_p(\log\jap{\x})=:b(x,\x)+b_1(x,\x).
\end{align}
Since $H_p(\log\jap{\x})\in S^0$, we have $b_1\in S^0$ and hence $[P,iA]-\Op(b)\in \Psi^{0}$. Although $\Op(b)\in \Psi^{+0}\setminus \Psi^0$, Lemma \ref{Mourrelow} shows $\Op(b)(H-i)^{-1}\in B(L^2)$ since $H_pk$ is supported away from $\left(\L_{\mathrm{inc}}\cup \L_{\mathrm{out}}\right)\cap  \{|\x|\geq 1\}$. Then we conclude $[P,A](H-i)^{-1}\in B(L^2)$. Moreover, since $V\in \Psi^{m-\k}$ with $\k<1$, $[V,A]\in B(L^2)$. These show $[H,A](H-i)^{-1}\in B(L^2)$.

Next, we show that $e^{itA}$ preserves $D(H)$.
By Lemma \ref{Schpreserve} and Taylor's theorem, the identity
\begin{align*}
He^{itA}=e^{itA}H+te^{itA}[H,iA]+\int_0^t(t-s)e^{i(t-s)A}[[H,iA],iA]  e^{isA}ds
\end{align*}
holds on $C^{\infty}(M)$. Since $[H,A](H-i)^{-1},[[H,A],A]\in B(L^2)$, we have
\begin{align}\label{HeitAineq}
\|He^{itA}u\|_{L^2}\leq C(\|u\|_{L^2}+\|Hu\|_{L^2})\quad u\in C^{\infty}(\tor)
\end{align}
with $C>0$. If $u\in D(H)$, then the essential self-adjointness of $H$ on $C^{\infty}(\tor)$ shows the existence of $u_n\in C^{\infty}(\tor)$ such that $u_n\to u$ and $Hu_n\to Hu$ as $n\to \infty$. By Lemma \ref{Schpreserve} and \eqref{HeitAineq}, $e^{itA}u_n\in C^{\infty}(\tor)$ and the sequence $\{He^{itA}u_n\}_n$ is Cauchy in $L^2$. These imply $e^{itA}u\in D(H)$, which shows that $e^{itA}$ preserves $D(H)$. 

Applying Lemma \ref{C^1cri2}, we conclude $H\in C^2(A)$.

\end{proof}

Finally, we prove the Mourre inequality.

\begin{lem}\label{Mouineq}
Let $I\Subset \re$. Then there exists $c>0$ such that
\begin{align*}
E_I(H)[H,iA]E_I(H)\geq cE_I(H)+K,
\end{align*}
where $K$ is a compact operator on $L^2$.

\end{lem}

\begin{proof}
Since $[V,A]$ is a negative order psedodifferential operator, it is a compact operator on $L^2$. Therefore, it suffices to prove an inequality where $[H,iA]$ is replaced by $[P,iA]$.

Set $\ell_m(\x)=\jap{\x}^{1-m}$ for $0< m<1$ and $\ell_m(\x)=\log\jap{\x}$ for $m=1$. 
Similar to \eqref{conjder}, we write $H_pa=-kH_p\ell_m+b$, where $b\in S^{+0}$ which is supported away form $\left(\L_{\mathrm{inc}}\cup \L_{\mathrm{out}}\right)\cap \{ \pm \x\geq 1\}$. We write $p(x,\x)=a_{\pm}(x)|\x|^m$ for $\pm\x\geq 1$ as in Lemma \ref{prilem}. Then there exists $c>0$ such that
\begin{align*}
-kH_pc_m(x,\x)=k(x,\x)a_{\pm}'(x)|\x|^m\pa_{\x}\ell_m(\x) \geq c
\end{align*}
near $(x,\x)\in \left(\L_{\mathrm{inc}}\cup \L_{\mathrm{out}}\right)\cap \{|\x|\geq 1\}$ due to the choice of $k$ in Lemma \ref{escexist}. Therefore, we have $H_pa\geq c+ b_3$ everywhere, where $b_2\in S^{+0}$ is supported away form $\left(\L_{\mathrm{inc}}\cup \L_{\mathrm{out}}\right)\cap \{|\x|\geq 1\}$. By the sharp G\aa rding inequality, we have
\begin{align*}
[P,iA]\geq c+\Op(b_2)+K_1,
\end{align*}
where $K_1\in \Psi^{-1+0}$ is a compact operator. By Lemma \ref{Mourrelow}, the operator $\Op(b_2)E_I(H)$ is bounded and compact on $L^2$. Then $E_I(H)[P,iA]E_I(H)\geq cE_I(H)+K$ holds with a compact operator $K=E_I(H)(\Op(b_2)+K)E_I(H)$. This completes the proof.

\end{proof}

\begin{proof}[Proof of Theorem \ref{spthm} $(i)$]
The result follows from Theorem \ref{Mourrethm}, Lemmas \ref{P+VC^2} and \ref{Mouineq}.

\end{proof}

\subsection{Discreteness of the spectrum}\label{subsecdiscspetc}

In this subsection, we prove Theorem \ref{spthm} $(ii)$. Let $P,V$ be satisfying the assumption of Theorem \ref{mainess} and suppose that $m>1$ holds and $p$ is not elliptic. To prove the discreteness of the spectrum, we need the following type of the subelliptic estimate.

\begin{prop}\label{propradsink}
For each $\d>0$ and $N>0$, then there exists $C>0$ such that
\begin{align}\label{radsinkpropes}
\|u\|_{H^{\frac{m-1}{2}-\d}}\leq C\|(P+V)u\|_{H^{-\frac{m-1}{2}-\d}}+C\|u\|_{H^{-N}}.
\end{align}
for $u\in H^{-N}$ with $(P+V)u\in H^{-\frac{m-1}{2}-\d}$. In particular, if $u\in L^2$ with $(P+V)u\in L^2$, then we have $u\in H^{\frac{m-1}{2}-\d}$.

\end{prop}

\begin{rem}
This estimate is an analog (or a stronger form) of \cite[Proposition 5.2]{T} although it holds only on the minimal domain of the operator when the higher dimensional cases are concerned. To prove the estimate on the maximal domain (that is, for functions satisfying $u\in L^2$ with $(P+V)u\in L^2$), we need to use the fact that the characteristic set of $p$ consists of the radial source and the radial sink. In particular, the radial source and sink sets are isolated. Taking it into account, we can use the radial sink estimate (\cite[Theorem E.54]{DZ}) with the vanishing propagation term (more precisely, the fact that the wavefront set of $B_1$ appearing in \cite[Theorem E.54]{DZ} is contained in the elliptic set in our case) instead. Here, we give the self-contained proof.
\end{rem}

\begin{proof}
We may assume $\k>0$. Let $k\in S^0$ is as in Lemma \ref{escexist}.
Set 
\begin{align*}
b(x,\x)=&k(x,\x)\jap{\x}^{-2\d}\in S^{-2\d},\quad \l_{\e}=\jap{\e \x}^{-\ell},\quad b_{\e}=b\l_{\e}^2,\\
B=&\Op(b),\quad B_{\e}=\Op(b_{\e}),\quad  \L^j=(I-\Delta)^{\frac{j}{2}},  \quad \L_{\e}=(1-\e\Delta)^{-\frac{\ell}{2}},\quad 
\end{align*}
for $0<\e\leq 1$ and $\ell>0$. We remark that the estimates given in the following are uniform in $0<\e\leq 1$ but are not needed to be uniformity with respect to $\ell>0$.

First, we claim that there exists $c>0$ such that
\begin{align}\label{sinkopineq}
[P+V,iB_{\e}]\geq c\L_{\e}\jap{D_x}^{m-1-2\d}\L_{\e}+\L_{\e}S_0\L_{\e}-C_0\L_{\e}\jap{D_x}^{m-1-\k-2\d}\L_{\e},
\end{align}
as a quadratic form on $C^{\infty}(\tor)$, where $B_{\e}=\Op(b_{\e})$, $\L_{\e}=\Op(\l_{\e})$, $S_0=\Op(s_0)$ and $s_0\in S^{m-1-2\d}$ is supported away from $(\L_{\mathrm{inc}}\cup \L_{\mathrm{out}})\cap \{|\x|\geq 1\}$ (see Subsection \ref{escapesubst} for the definition of $\L_{\mathrm{inc}},\L_{\mathrm{out}}$).
Since $H_p\jap{\x}^{-2\d}=-a_{\pm}'(x)|\x|^m\pa_{\x}\jap{\x}^{-2\d}=2\d a_{\pm}'(x)|\x|^m\x \jap{\x}^{-2\d-2}$ by \eqref{prepresent}, there is $c>0$ such that $H_p\jap{\x}^{-2\d}\geq c\jap{\x}^{m-1-2\d}$ if $(x,\x)$  
belongs to a neighborhood of $(\L_{\mathrm{inc}}\cup \L_{\mathrm{out}})\cap \{|\x|\geq 1\}$. Then Lemma \ref{escexist} shows $H_pb(x,\x)\geq c\jap{\x}^{m-1-2\d}+s_0$, where $s_0\in S^{m-1-2\d}$ is supported away from $(\L_{\mathrm{inc}}\cup \L_{\mathrm{out}})\cap \{|\x|\geq 1\}$.
By a similar calculation, we have $k(x,\x)H_p\jap{\e \x}^{-\ell}\geq 0$ (due to $\ell>0$) and hence
\begin{align}\label{radsinkclin}
\l_{\e}(\x)^{-2}H_pb_{\e}(x,\x)\geq c\jap{\x}^{m-1-2\d}+s_0(x,\x).
\end{align}

Since $V\in \Psi^{m-\k}$ with $0<\k<1$ and $B\in \Psi^{-2\d}$, we have
\begin{align}\label{Vlowerest}
\L_{\e}^{-1}[V,iB_{\e}]\L_{\e}^{-1}\geq -C_0'\L^{m-1-\k-2\d}
\end{align}
holds with a constant $C_0'>0$ independent of $0<\e\leq 1$.
By the sharp G\aa riding inequality, \eqref{radsinkclin}, and \eqref{Vlowerest}, there is $C_0>0$ independent of $0<\e\leq 1$ such that $\L_{\e}^{-1}[P+V,iB_{\e}]\L_{\e}^{-1}\geq c\L^{m-1-2\d}+S_0-C_0\L^{m-1-\k-2\d}$, which shows \eqref{sinkopineq}.

The operator inequality \eqref{sinkopineq} implies
\begin{align}
c\|\L_{\e}u\|_{H^{\frac{m-1}{2}-\d}}^2\leq& |(u,[P+V,iB_{\e}]u)_{L^2}|+|(\L_{\e}u,S_0\L_{\e}u)_{L^2}|+C_0\|\L_{\e}u\|_{H^{\frac{m-1-\k}{2}-\d}}^2\label{radsiquin1}
\end{align}
for $u\in C^{\infty}(\mathbb{T})$. By the AM-GM inequality,
\begin{align}
|(u,[P+V,iB_{\e}]u)_{L^2}|\leq& 2|(B\L_{\e}(P+V)u,\L_{\e}u)_{L^2}|\nonumber\\
\leq&2\|B\L_{\e}(P+V)u\|_{H^{-\frac{m-1}{2}+\d}}  \|\L_{\e}u\|_{H^{\frac{m-1}{2}-\d}}\nonumber\\
\leq&\frac{2}{c}\|B\L_{\e}(P+V)u\|_{H^{-\frac{m-1}{2}+\d}}^2+\frac{c}{2}  \|\L_{\e}u\|_{H^{\frac{m-1}{2}-\d}}^2\nonumber\\
\leq&C'\|\L_{\e}(P+V)u\|_{H^{-\frac{m-1}{2}-\d}}^2+\frac{c}{2}  \|\L_{\e}u\|_{H^{\frac{m-1}{2}-\d}}^2\label{radsiquin2}
\end{align}
where $C'>0$ is a constant and we use $B\in \Psi^{-2\d}$ in the last line. On the other hand, since  $s_0\in S^{m-1-2\d}$ is supported away from $(\L_{\mathrm{inc}}\cup \L_{\mathrm{out}})\cap \{|\x|\geq 1\}$, the elliptic estimate (or the usual parametrix construction) implies
\begin{align}\label{radsiquin3}
|(\L_{\e}u,S_0\L_{\e}u)_{L^2}|\leq C''\|\L_{\e}(P+V)u  \|_{H^{-\frac{m-1}{2}-\d}(\mathbb{T})}^2   +C''\|\L_{\e}u\|_{H^{-N}}^2
\end{align}
with a constant $C''>0$ independent of $0<\e\leq 1$. Moreover, the interpolation inequality gives
\begin{align}\label{radsiquin4}
C_0\|\L_{\e}u\|_{H^{\frac{m-1-\k}{2}-\d}}^2\leq \frac{c}{4}\|\L_{\e}u\|_{H^{\frac{m-1}{2}-\d}}^2+ C'''\|\L_{\e}u\|_{H^{-N}}^2.
\end{align}
By \eqref{radsiquin1}, \eqref{radsiquin2}, \eqref{radsiquin3}, and \eqref{radsiquin4}, we obtain
\begin{align}\label{radsiquin5}
\|\L_{\e}u\|_{H^{\frac{m-1}{2}-\d}}^2\leq C\|\L_{\e}(P+V)u  \|_{H^{-\frac{m-1}{2}-\d}}^2+C\|\L_{\e}u\|_{H^{-N}}^2.
\end{align}
for $u\in C^{\infty}(\mathbb{T})$, where we set $C:=\frac{4}{c}\max(C'+C'',C''+C''')$ which is independent of $0<\e\leq 1$. This proves \eqref{radsinkpropes} for $u\in C^{\infty}(\mathbb{T})$ taking $\e\to 0$.

Let $u\in H^{-N}$ with $(P+V)u\in H^{-\frac{m-1}{2}-\d}$. Taking $\ell\geq \frac{m-1}{2}+N-\d$, we have
\begin{align*}
\L_{\e}\in B(H^{-N}, H^{\frac{m-1}{2}-\d}),\quad \L_{\e}(P+V)\in B(H^{-N},H^{-\frac{m-1}{2}-\d})
\end{align*}
since $\L_{\e}\in \Psi^{-\ell}$ and $\L_{\e}(P+V)\in \Psi^{m-\ell}$ for fixed $\ell>0$. Since $C^{\infty}(\mathbb{T})$ is dense in $H^{-N}$, there exists $u_n\in C^{\infty}(\mathbb{T})$ such that $u_n\to u$ in $H^{-N}$. Then we have
$\L_{\e}u_n\to u$ in $H^{\frac{m-1}{2}-\d}$ and $\L_{\e}(P+V)u_n\to \L_{\e}(P+V)u$ in $H^{-\frac{m-1}{2}-\d}$ as $n\to \infty$. Inserting $u_n$ into \eqref{radsiquin5} and taking $n\to \infty$, we conclude that \eqref{radsiquin5} is true for such $u$. Next, taking $\e\to 0$, we obtain \eqref{radsinkpropes}. This completes the proof.

\end{proof}

\begin{rem}
The Yosida approximation $\L_{\e}$ is not needed to prove the inequality \eqref{radsinkpropes} for $u\in C^{\infty}(\tor)$ but is needed for more general classes of functions satisfying $u\in H^{-N}$ with $(P+V)u\in H^{-\frac{m-1}{2}-\d}$.
\end{rem}

\begin{proof}[Proof of Theorem \ref{spthm} $(ii)$]

We recall from \eqref{Dmaxdef} that $D_{\mathrm{max}}(P+V)=\{u\in L^2\mid (P+V)u\in L^2\}$ is the maximal domain of $((P+V|_{C^{\infty}(\tor)})^*$.
Then \eqref{radsinkpropes} implies that the Banach space $D_{max}(P+V)$ with its graph norm is included in $H^{\frac{m-1}{2}-\d}$ continuously for each $\d>0$. We also note that if $\d<\frac{m-1}{2}$, then $H^{\frac{m-1}{2}-\d}$ is compactly embedded into $L^2$ by the Rellich-Kondrachov theorem. Combining them, we conclude that the natural inclusion map $D_{\mathrm{max}}\hookrightarrow L^2$ is a compact operator.

Now fix a self-adjoint extension of $P+V$. We write its domain by $D$. By virtue of \cite[Proposition 5.1]{T}, it suffices to prove that the natural inclusion $D\subset L^2$ (where we regard $D$ as a Banach space with the graph norm) is compact for proving Theorem \ref{spthm} $(ii)$. This follows from the continuity of $D\hookrightarrow D_{\mathrm{max}}$ and the compactness of $D_{max}\hookrightarrow L^2$, which completes the proof. 
\end{proof}

\appendix

\section{Anisotropic Sobolev space}

In this appendix, we recall some fundamental properties of symbols with variable order, based on \cite[Appendix]{FRS}. 
Let $(M,g)$ be a closed Riemannian manifold.
For a real valued symbol $k\in S^0$ and $\rho\in (\frac{1}{2},1)$, we set
\begin{align*}
g_{k}(x,\x)=\jap{\x}^{k(x,\x)},\quad S_{\rho}^{k(x,\x)}:=\{a\in C^{\infty}(T^*M)\mid |\pa_{x}^{\a}\pa_{\x}^{\b}a(x,\x)|\leq C\jap{\x}^{k(x,\x)+(1-\rho)|\a|-\rho|\b|}\}.
\end{align*}

\begin{lem}\label{Anistrolem}

\noindent$(i)$ Set $k_{\mathrm{max}}:=\max_{(x,\x)\in T^*M}k(x,\x)$. For $B\in \Op S_{\rho}^{k(x,\x)}$ with $\rho\in (\frac{1}{2},1)$, we have $B\in B(H^s,H^{s-k_{\mathrm{max}}})$.

\noindent$(ii)$ There exists $G_k\in \bigcap_{1/2<\rho<1}\Op S^{k(x,\x)}_{\rho}$ such that 
\begin{align*}
G_k-\Op(g_k)\in \bigcap_{1/2<\rho<1}\Op S^{k(x,\x)-2\rho+1}_{\rho},
\end{align*}
and $G_k$ is formally self-adjoint and invertible both on $C^{\infty}(M)$ and $\mathcal{D}'(M)$. Moreover, $G_k^{-1}\in \bigcap_{1/2<\rho<1}\Op S^{-k(x,\x)}_{\rho}$ and $G_k^{-1}-\Op(g_k^{-1})\in \bigcap_{1/2<\rho<1}\Op S^{-k(x,\x)-2\rho+1}_{\rho}$.

\noindent$(iii)$ Let $\L$ be a closed conic subset of $T^*M\setminus 0$. Suppose that there exists $\e_0\in \re$ such that $k(x,\x)\geq \e_0$ in a conic neighborhood $U$ of $\L$. Then 
\begin{align*}
u\in\mathcal{D}'(M),\,\, G_ku\in H^{\ell}\Rightarrow \mathrm{WF}^{\ell+\e_0}(u)\cap \L=\emptyset
\end{align*}

\noindent$(iv)$ For $m\in \re$ and $A\in \Psi^m$,
 \begin{align*}
G_kAG_k^{-1}=A+i\Op(H_{a}(k\log\jap{\x}))+\Psi^{m-2+0},
\end{align*}
where $\Psi^{\ell+0}:=\cap_{\e>0}\Psi^{\ell+\e}$.
\end{lem}

\begin{proof}
The part $(i)$, $(ii)$ follows from \cite[A.3.4]{FRS} and \cite[Lemma 12, Corollary 4]{FRS} respectively.
The proof of the part $(iv)$ is same as in \cite[Lemma 3.2]{FS}. We just write $G_kAG_k^{-1}=A+i[A,iG_k]G_{k}^{-1}$.

We briefly discuss the proof of $(iii)$. Set $G_ku=v\in H^{\ell}$ and write $u=G_k^{-1}v$. Take $b\in S^0$ such that $b=1$ on $\L\cap \{|\x|_g\geq 1\}$ and $\supp b\subset U$.
 Since the principal part of the symbol of $G_k^{-1}$ is $g_k^{-1}=\jap{\x}^{-k(x,\x)}$, our assumptions imply $BG_k^{-1}(I-\Delta_g)^{\frac{\e_0}{2}}\in \Psi^0$ and hence it preserves $H^{\ell+\e_0}$. Therefore, due to the ellipticity of $(I-\Delta_g)$, we have $(I-\Delta_g)^{-\frac{\e_0}{2}}v\in H^{\ell+\e_0}$ and hence $Bu=BG_k^{-1}(I-\Delta_g)^{\frac{\e_0}{2}}\cdot (I-\Delta_g)^{-\frac{\e_0}{2}} v\in H^{\ell+\e_0}$.

\end{proof}

\section{Subprincipal symbol in a standard coordinate}

As a general result, we have the following proposition. See \cite[Theorem 18.1.34]{Ho} for pseudodifferential operators acting on half-densities.

\begin{prop}\label{prireal}
Let $P\in \Psi^m_{\mathrm{phg}}$ and $d\m$ be a smooth density on $M$. In a coordinate patch, we write the total symbol of $P$ by $p_{\mathrm{tot}}\sim \sum_{j=0}^{\infty}p_{m-j}$ with $p_{m-j}$ homogeneous of order $m-j$ and $d\m=\m(x)dx$, where $dx$ is the Lebesgue measure on that coordinate. Suppose that $P$ is symmetric on $C^{\infty}(M)$ with respect to the metric on $L^2(d\m)=L^2(M;d\m)$. Then $p_m=\s(P)$ is real-valued and
\begin{align*}
\Im \s_{\mathrm{sub}}(P)=-\frac{1}{2\m(x)}\sum_{j=1}^n\pa_{x_j}\m(x)\pa_{\x_j}p_m(x,\x),
\end{align*}
where we recall $\s_{\mathrm{sub}}(P)=p_{m-1}+\frac{i}{2}\sum_{j=1}^n\pa_{x_j}\pa_{\x_j}p_m$.

\end{prop}

\begin{proof}
Since the problem is local, we may assume $M=\re^n$ and $\m\in C_b^{\infty}(\re^n)$. We denote $P^*$ by the adjoint of $P$ with respect to the Lebesgue measure $dx$ and write $P^*=p_{\mathrm{tot}}^*(x,D_x)$. 
The adjoint operator of $P$ with respect to $P$ is $\m^{-1}P^* \m$ since $(Pu,w)_{L^2(d\m)}=(\m Pu,w)_{L^2(dx)}=(u, P^* \m w)_{L^2(dx)}=(u, \m^{-1}P^* \m w)_{L^2(d\m)}$. Thus our assumption implies $P= \m^{-1}P^* \m$. 

We compute the symbol of $\m^{-1}P^*\m$. This operator is a pseudodifferential operator with a three variable symbol, that is, $\m^{-1}P^*\mu(x)=\frac{1}{(2\pi)^n}\int_{\re^n}\int_{\re^n}a(x,y,\x)e^{i(x-y)\cdot \x}u(y)dyd\x$, where $a(x,y,\x)=\m(x)^{-1}p_{\mathrm{tot}}^*(x,\x)\m(y)$ that is of order $m$. Then it is known that there exists $b\in S^m$ such that $\m^{-1}P\m=b(x,D_x)$ and $b(x,\x)\sim \sum_{\a\in\mathbb{Z}_{\geq 0}^n}\left.\frac{i^{|\a|}}{\a!}D_{y}^{\a}D_{\x}^{\a}a(x,y,\x)\right|_{y=x}$. Combining this with the well-known asymptotic formula $p_{\mathrm{tot}}^*(x,\x)\sim \sum_{\a\in\mathbb{Z}_{\geq 0}^n}\frac{i^{|\a|}}{\a!}D_{x}^{\a}D_{\x}^{\a}\overline{p_{\mathrm{tot}}}(x,\x)$, we have
\begin{align*}
b(x,\x)=&p^*(x,\x)-i\m(x)^{-1}\sum_{j=1}^n\pa_{x_j}\m(x)\pa_{\x_j}p^*(x,\x)\\
=&\overline{p}(x,\x)-i\sum_{j=1}^n\pa_{x_j}\pa_{\x_j}\overline{p}(x,\x)   -i\m(x)^{-1}\sum_{j=1}^n\pa_{x_j}\m(x)\pa_{\x_j}\overline{p}(x,\x)\\
=&\overline{p_m}(x,\x)+\overline{p_{m-1}}(x,\x)-i\sum_{j=1}^n\pa_{x_j}\pa_{\x_j}\overline{p_m}(x,\x)   -i\m(x)^{-1}\sum_{j=1}^n\pa_{x_j}\m(x)\pa_{\x_j}\overline{p_m}(x,\x)
\end{align*}
modulo $S^{m-2}$. By the relation $b(x,D_x)=\m P^*\m=P=p_{\mathrm{tot}}(x,D_x)$, we have $p_{\mathrm{tot}}=b$ and hence
\begin{align*}
p_m=\overline{p_m},\quad p_{m-1}-\overline{p_{m-1}}=-i\sum_{j=1}^n\pa_{x_j}\pa_{\x_j}p_m(x,\x)   -i\m(x)^{-1}\sum_{j=1}^n\pa_{x_j}\m(x)\pa_{\x_j}p_m(x,\x),
\end{align*}
which completes the proof.
\end{proof}

\begin{cor}\label{stdsubprireal}
Let $P\in \Psi^m_{\mathrm{phg}}$ and $dx$ be the flat density on the torus $\tor$. In a standard coordinate (see Definition \ref{stdcoord}), we write the total symbol of $P$ by $p_{\mathrm{tot}}\sim \sum_{j=0}^{\infty}p_{m-j}$ with $p_{m-j}$ homogeneous of order $m-j$. If $P$ is symmetric on $C^{\infty}(\tor)$ with respect to the metric on $L^2(\tor)=L^2(\tor;dx)$. Then $p_m(=\s(P))$ and $\s_{\mathrm{sub}}(P)$ are real-valued.
\end{cor}

\begin{proof}
We just apply the last proposition with $\m(x)=1$.

\end{proof}

\end{document}